\documentclass[11pt]{article}

\usepackage{amsthm, amsmath, amssymb, amsfonts, url, booktabs, tikz, setspace, fancyhdr, bm}
\definecolor{seedblue}{RGB}{67,112,156}
\usepackage{cancel}
\usepackage{geometry}
\usepackage{hyperref, enumerate}
\usepackage[shortlabels]{enumitem}
\usepackage[babel]{microtype}
\usepackage[english]{babel}
\usepackage[capitalise]{cleveref}
\usepackage{comment}
\usepackage{bbm}
\usepackage{csquotes}
\usepackage{mathabx}
\usepackage{tikz}
\usepackage{subcaption}
\usepackage{graphicx}
\usepackage{float}
\usepackage{xcolor}
\usepackage[normalem]{ulem}
\usepackage{diagbox}
\usetikzlibrary{positioning, arrows.meta, shapes.geometric}

\counterwithin{figure}{section}

\newtheorem{theorem}{Theorem}[section]
\newtheorem{prop}[theorem]{Proposition}
\newtheorem{lemma}[theorem]{Lemma}
\newtheorem{cor}[theorem]{Corollary}

\newtheorem{claim}[theorem]{Claim}

\newtheorem{fact}[theorem]{Fact}

\theoremstyle{definition}

\newtheorem{definition}[theorem]{Definition}
\newtheorem*{defn-non}{Definition}

\newtheorem{ques}[theorem]{Question}
\definecolor{rosepink}{RGB}{255,102,204}
\definecolor{dateplum}{HTML}{993366}
\definecolor{darkdateplum}{RGB}{128,0,32}
\definecolor{lightdateplum}{RGB}{219,112,147}
\definecolor{darkred}{RGB}{139,0,0}
\definecolor{lightred}{RGB}{240,130,100}

\newlist{Case}{enumerate}{3}
\setlist[Case, 1]{%
    label           =   {\bfseries Case \arabic*.},
    labelindent=1em ,labelwidth=1cm, labelsep*=1em, leftmargin =!
}
\setlist[Case, 2]{%
    label           =   {\bfseries Subcase \arabic{Casei}.\arabic*.},
    labelindent=-1em ,labelwidth=1cm, labelsep*=1em, leftmargin =!
}
\setlist[Case, 3]{%
    label           =   {\bfseries Subsubcase \arabic{Casei}.\arabic{Caseii}.\arabic*.},
    labelindent=-1em ,labelwidth=1cm, labelsep*=1em, leftmargin =!
}

\newenvironment{poc}{\begin{proof}[Proof of claim]}{\end{proof}}

\newcommand{\cD}{{\mathcal D}}
\newcommand{\cC}{{\mathcal C}}
\newcommand{\cT}{{\mathcal T}}
\usepackage{todonotes}

\newcommand*{\abs}[1]{\lvert#1\rvert}

\title{Asymptotics of the Brown--Erd\H{o}s--S\'os problem at integer exponents}

\author{
Ting-Wei Chao\thanks{Department of Mathematics, Massachusetts Institute of Technology, Cambridge, MA, USA. Email: {\tt twchao@mit.edu}}
\and
Xinqi Huang\thanks{School of Mathematical Sciences, University of Science and Technology of China, Hefei, China and Extremal Combinatorics and Probability Group (ECOPRO), Institute for Basic Science (IBS), Daejeon, South Korea.
Email: \texttt{huangxq@mail.ustc.edu.cn}. Supported by the USTC Excellent PhD Students Overseas, the Institute for Basic Science (IBS-R029-C4), the National Key Research and Development Programs of China 2023YFA1010200, the NSFC under Grants No. 12171452 and No. 12231014 and Innovation Program for Quantum Science and Technology 2021ZD0302902.
}
\and
Hong Liu\thanks{Extremal Combinatorics and Probability Group (ECOPRO), Institute for Basic Science (IBS), Daejeon, South Korea. Emails: \texttt{hongliu@ibs.re.kr}. Supported by the Institute for Basic Science (IBS-R029-C4).}
}

\AddToHook{env/lemma/begin}{\crefalias{theorem}{lemma}}
\AddToHook{env/prop/begin}{\crefalias{theorem}{proposition}}
\AddToHook{env/cor/begin}{\crefalias{theorem}{corollary}}
\AddToHook{env/claim/begin}{\crefalias{theorem}{claim}}
\AddToHook{env/fact/begin}{\crefalias{theorem}{fact}}
\AddToHook{env/definition/begin}{\crefalias{theorem}{definition}}
\AddToHook{env/defn/begin}{\crefalias{theorem}{definition}}
\AddToHook{env/ques/begin}{\crefalias{theorem}{question}}
\AddToHook{env/rmk/begin}{\crefalias{theorem}{remark}}
\crefname{claim}{Claim}{Claims}
\Crefname{claim}{Claim}{Claims}
\crefname{fact}{Fact}{Facts}
\Crefname{fact}{Fact}{Facts}
\crefname{question}{Question}{Questions}
\Crefname{question}{Question}{Questions}

\begin{document}
\date{}
\maketitle

\begin{abstract}
The Brown--Erd\H{o}s--S\'os problem is a fundamental problem in
sparse hypergraph Tur\'an theory.
For integers $r,k\ge 2$ and $s\ge r$, the problem asks for the maximum number $f^{(r)}(n;s,k)$ of edges
in an $n$-vertex $r$-uniform hypergraph containing no $k$ distinct
edges spanning at most $s$ vertices.
In 1971, Brown, Erd\H{o}s, and S\'os proved that $f^{(r)}\bigl(n;(r-t)k+t,k\bigr)=\Theta(n^t)$ for all $r>t\ge 2$
and $k\ge 2$. However,
the existence and the value of the
leading coefficient have remained largely open.

We determine the coefficient
$\pi(r,t,k):=\lim_{n\to\infty}n^{-t}
f^{(r)}\bigl(n;(r-t)k+t,k\bigr)$ for every such $r,t,k$, except when $(r,t)=(3,2)$ and $k\ge 4$ is even. In particular, 
\[
\pi(r,t,k)
=
\begin{cases}
\frac{2}{t!\bigl(2\binom{r}{t}-1\bigr)},
& \text{if $k$ is odd},\\
\frac{1}{t!\binom{r}{t}},
& \text{if $k$ is even and $r\ge4$}.
\end{cases}
\]
Surprisingly, for $r\ge4$, the limit $\pi(r,t,k)$ depends on $k$ only through its parity, not its value. 
For the remaining case, we prove that
$\pi(3,2,k)>1/6$ for all even $k\ge 4$, showing that the natural packing construction is never
asymptotically optimal.
As an application, we extend a connection of Bennett,
Cushman, and Dudek to arbitrary uniformity to resolve several cases of the Erd\H{o}s--Gy\'arf\'as--Shelah generalized Ramsey problem.
\end{abstract}

\section{Introduction}
How many $r$-element subsets of $[n]:=\{1,2,\ldots,n\}$ can one choose if every $k$ of them together contain more than $s$ elements?
This is the Brown--Erd{\H{o}}s--S\'os problem, a central example of
a sparse hypergraph Tur\'an problem; see~\cite{keevash_hypergraph_turan_2011}
for the broader context. For integers $r,k\ge2$ and $s\ge r$,
we write $f^{(r)}(n;s,k)$ for the maximum number of edges in an
$n$-vertex $r$-graph containing no $k$ distinct edges whose union
has at most $s$ vertices.
Brown, Erd{\H{o}}s, and S\'os~\cite{brown_erdos_sos_1973} established in the 70s the
general bounds
\[
    \Omega\left(
        n^{\frac{rk-s}{k-1}}
    \right)
    \le
    f^{(r)}(n;s,k)
    \le
    O\left(
        n^{\left\lceil\frac{rk-s}{k-1}\right\rceil}
    \right)
\]
for fixed $r,k\ge2$ and $r\le s\le rk$.

When $(rk-s)/(k-1)$ is not an integer, these bounds do not generally
determine even the correct order of magnitude. The most prominent
example is the classical Brown--Erd{\H{o}}s--S\'os conjecture, which
asserts that
$f^{(3)}(n;k+3,k)=o(n^2)$
for every fixed $k\ge3$
\cite{brown_erdos_sos_1973}.
Its first case is the celebrated $(6,3)$-problem, solved by
Ruzsa and Szemer\'edi
\cite{ruzsa_szemeredi_1978},
whereas the next case, the $(7,4)$-problem, remains open.
For subsequent progress on approximate forms of the conjecture, see
\cite{conlon_gishboliner_levanzov_shapira_2023,
janzer_methuku_milojevic_sudakov_2025}; for progress in the dense
linear setting, see
\cite{gishboliner_solymosi_counting_2026,santos_tyomkyn_dense_2025}.

Another fundamental direction concerns the case in which
$t=(rk-s)/(k-1)$ is an integer with $1\le t<r$.
The case $t=1$ is elementary and completely understood
\cite{glock_64_2024}. In particular, 
    $\lim_{n\to\infty}
    \frac{
        f^{(r)}\bigl(n;(r-1)k+1,k\bigr)
    }{n}=\frac{k-1}{(k-1)(r-1)+1}.$
We therefore focus on the case $2\le t<r$.
The preceding bounds by Brown, Erd{\H{o}}s, and S\'os give the exponent:
$f^{(r)}\bigl(n;(r-t)k+t,k\bigr)=\Theta(n^t)$. Thus the problem is to determine the leading
coefficient. More precisely, does the limit
\[
    \pi(r,t,k):=
    \lim_{n\to\infty}
    \frac{
        f^{(r)}\bigl(n;(r-t)k+t,k\bigr)
    }{n^t}
\]
exist, and what is its value? Brown, Erd{\H{o}}s, and S\'os
\cite{brown_erdos_sos_1973}
conjectured that $\pi(3,2,k)$ exists for every $k\ge2$.
The conjecture was proved by Delcourt and Postle
\cite{delcourt_postle_limit_2024}, and Shangguan~\cite{shangguan_degenerate_2023} subsequently
established the existence of $\pi(r,2,k)$ for every $r\ge3$ and
$k\ge2$.

Despite this progress, exact coefficients were known only in a few
cases. For arbitrary $r>t\ge2$, R\"odl determined the case $k=2$
\cite{rodl_packing_1985}, while Glock, Joos, Kim, K\"uhn, Lichev, and
Pikhurko determined $k\in\{3,4\}$
\cite{glock_64_2024}. In the original case $(r,t)=(3,2)$, the exact
values were known for $2\le k\le7$
\cite{glock_triple_2019,glock_64_2024,glock_k567_2025},
with further results only in restricted parameter ranges
\cite{letzter_sgueglia_2025,pikhurko_sun_2026,wang_zeng_2026}.
Thus no general formula for $\pi(r,t,k)$ was known, and for
$t\ge3$ even its existence remained open in general
\cite{letzter_sgueglia_2025,shangguan_tamo_2020}.

Our main result determines the asymptotics for every odd $k$, and for every even $k$ when $r\ge4$. 
Together with the known result
for triple systems, it resolves the general existence problem: $\pi(r,t,k)$ exists for all integers $r>t\ge2$ and $k\ge2$.

\begin{theorem}\label{thm:BES-r-k-t-odd}
    For all integers $r>t\ge2$ and $k\ge 2$, we have
    \[
\pi(r,t,k)
=
\begin{cases}
\frac{2}{t!\bigl(2\binom{r}{t}-1\bigr)},
& \text{if $k$ is odd},\\
\frac{1}{t!\binom{r}{t}},
& \text{if $k$ is even and $r\ge4$}.
\end{cases}
\]
\end{theorem}

The two formulas reveal a striking parity phenomenon. For every
fixed $r\ge4$ and $r>t\ge2$, the limit $\pi(r,t,k)$ depends on
$k$ only through its parity, and not on the particular value of $k$.
For example, $\pi(4,2,k)$ alternates between $1/11$ for odd $k$
and $1/12$ for even $k$.

The only cases not covered by \cref{thm:BES-r-k-t-odd} are
$\pi(3,2,k)$ for even $k$.
Here $\pi(3,2,2)=1/6$ is the packing coefficient, while
$\pi(3,2,4)=\frac{7}{36}$
and 
$\pi(3,2,6)=\frac{61}{330}$
were determined in
\cite{glock_64_2024,glock_k567_2025}
and are both strictly larger than $1/6$.
Our next result shows that this phenomenon persists for every even
$k\ge4$: the usual packing construction is never asymptotically optimal.

\begin{theorem}\label{thm:BES-3-2-even}
For every even $k\ge4$, we have $\pi(3,2,k)>
\frac16.$
\end{theorem}

\Cref{thm:BES-3-2-even} shows that this is a genuine
exceptional family, not merely a limitation of our method.
We return to it in \Cref{sec:conclusion}.

\subsection{An application to the Erd{\H{o}}s--Gy\'arf\'as--Shelah
problem}
\label{subsec:Erdos-Shelah-application}

Our results also have an application to a generalized Ramsey problem
introduced by Erd{\H{o}}s and Shelah
\cite{erdos_problems_1975}
and subsequently studied systematically, in the graph case, by
Erd{\H{o}}s and Gy\'arf\'as
\cite{erdos_gyarfas_1997}.
This problem is a natural extension of classical Ramsey theory. For integers $r\ge2$, $n\ge p\ge r+1$, and
$1\le q\le\binom{p}{r}$, a \emph{$(p,q)$-coloring} of
$K_n^{(r)}$ is a coloring
$\chi:\binom{[n]}{r}\longrightarrow\mathcal C$
such that, for every $S\in\binom{[n]}{p}$,
the edges in $\binom{S}{r}$ receive at least $q$ distinct colors.
We denote by
$\operatorname{ES}^{(r)}(n;p,q)$
the minimum number of colors used by a $(p,q)$-coloring of
$K_n^{(r)}$.

The case $q=2$ is the inverse form of the classical multicolor Ramsey
problem. More generally, a central question is to determine the order
of magnitude of
$\operatorname{ES}^{(r)}(n;p,q)$
for fixed $r,p,q$, and in particular to determine its polynomial
exponent. For an overview of this problem and related developments,
see~\cite[Section~7]{mubayi_suk_2020};
see also~\cite{
bennett_delcourt_li_postle_2026,bennett_dudek_english_2025,conlon_fox_lee_sudakov_2015}.

Since
$\operatorname{ES}^{(r)}(n;p,q)\le\binom{n}{r}$,
the largest possible polynomial exponent is $r$.
For fixed integers $r\ge2$ and $p\ge2r$, the first value of $q$ for
which this exponent is attained is
    $q_{\max}^{(r)}(p)
    :=
    \binom{p}{r}
    -
    \left\lfloor\frac{p}{r}\right\rfloor
    +
    2.$
More precisely,
$\operatorname{ES}^{(r)}(n;p,q_{\max}^{(r)}(p))
    =
    \Theta(n^r)$,
whereas
$\operatorname{ES}^{(r)}(n;p,q_{\max}^{(r)}(p)-1)   =o(n^r)$~\cite{bennett_delcourt_li_postle_2026,erdos_gyarfas_1997}.

We focus on the case where
$p=r(k+1)$ and
$q=\binom{p}{r}-k+1$, and $r,k\ge2$.
In this case, we have $q=q_{\max}^{(r)}(p)$.
When $r=2$, 
Bennett, Cushman, and Dudek
\cite{bennett_cushman_dudek_2025}
established an exact asymptotic connection between this family and the
quadratic Brown--Erd{\H{o}}s--S\'os problem. For every even integer
$p\ge6$, they proved that
    $\lim_{n\to\infty}
    \frac{
        \operatorname{ES}^{(2)}
        \bigl(n;p,q_{\max}^{(2)}(p)\bigr)
    }{n^2}
    =
    \frac12
    -
    \pi\left(4,2,\frac p2-1\right).$
Prior to the present work, in the graph case, the exact coefficient
at this threshold had been determined only for
$p\in\{6,8,10,12,14,16,18\}$
\cite{bennett_cushman_dudek_2025,glock_k567_2025,
pikhurko_sun_2026}.

By extending this connection
to arbitrary uniformity and combining it with our main results,
we obtain the following.

\begin{theorem}\label{thm:generalized-ramsey}
For all integers $r,k\ge2$, put
$p=r(k+1)$ and $q=\binom{p}{r}-k+1$. Then
\begin{align*}
    \lim_{n\to\infty}
    \frac{
        \operatorname{ES}^{(r)}(n;p,q)
    }{n^r}
    =
    \begin{cases}
        \frac{1}{r!}\left(1-\frac{2}{2\binom{2r}{r}-1}\right),
        & \text{if $k$ is odd},\\        \frac{1}{r!}\left(1-\frac{1}{\binom{2r}{r}}\right),
        & \text{if $k$ is even}.
    \end{cases}
\end{align*}
\end{theorem}

\subsection{Proof strategy}
\label{subsec:proof-strategy}
Both coefficients arise from counting $t$-sets. The lower
bounds seek to use these sets as efficiently as possible, while the
upper bounds show that greater efficiency would create a forbidden
configuration. 

\paragraph{Lower bounds: finding seeds and apply packing theorem.}
We begin with the lower-bound construction, where most of the technical difficulty lies.

\medskip
\noindent\emph{Use diamonds.}
A single edge contains $\binom rt$ different $t$-sets,
so packing edges without repeating a $t$-set suggests
$\binom nt/\binom rt$ edges. Two edges meeting in exactly $t$ vertices,
which we call a \emph{diamond}, contain only $2\binom rt-1$ different
$t$-sets. Packing diamonds therefore suggests the larger bound
$2\binom nt/(2\binom rt-1)$. The first bound is attained by known
high-girth partial Steiner systems~\cite{delcourt2022finding}.
Our construction attains the second when $k$ is odd. The obstacle
is that edges from different diamonds may together form a forbidden
configuration, even when their $t$-sets are not repeated.

\medskip
\noindent\emph{A baby example: reserve the extra pair only once.}
Consider the triples $abx$ and $aby$. They contain five pairs,
but not $xy$. A triple $xyz$, with $z$ new, repeats none of these
pairs, yet the three triples span only five vertices: a forbidden
configuration for $k=3$. To prevent this overlap, we could reserve
$xy$ as well, excluding it from other packed copies. But reserving
six pairs for two triples reduces the ratio from $2/5$ to $1/3$,
losing the improvement over Steiner triple systems. Instead, let many diamonds share the same extra pair. Fix $x,y$, take pairwise disjoint pairs $\{a_i,b_i\}$ disjoint from $\{x,y\}$,
and use
    $E_m=\{a_i b_i x,a_i b_i y:i\in[m]\}.$
We reserve all pairs contained in these triples, together with $xy$.
There are $2m$ triples and $5m+1$ reserved pairs, so the ratio tends
to $2/5$. The extra pair is reserved once for the whole block,
rather than once for each diamond.

The example also explains parity. Once $x,y$ are present, a whole
diamond contributes two edges and two new vertices; taking only
one of its triples contributes the same two vertices but just one
edge. In fact, a collection of $\ell\ge2$ triples in this block
spans at most $\ell+2$ vertices only if it is a union of whole
diamonds. Its size is then even, so the block contains no forbidden
configuration for odd $k$.

\medskip
\noindent\emph{From one block to the general construction.}
For larger $k$, more complicated configurations can involve several
packed blocks. Further $t$-sets must therefore be reserved to control
these interactions. 
A packing theorem of Glock, Joos, Kim, K\"uhn,
Lichev, and Pikhurko~\cite{glock_64_2024} reduces the lower bound to
constructing a suitable finite seed with a large packing ratio.
Previous lower-bound constructions in this framework relied on
finding suitable seeds for particular parameter ranges and
verifying their local properties
\cite{glock_64_2024,glock_k567_2025,pikhurko_sun_2026}. In our proof, we construct these seeds via a iterative probabilistic argument.

In each step, we take a blow-up of the current seed, randomly selects
copies so that many share the extra $t$-sets to be reserved,
and deletes those involved in unwanted configurations.
Parity is crucial to controlling the deletion losses:
in the critical overlap patterns, each participating copy must
contribute an even number of edges, so their union cannot be
an odd forbidden configuration.
This procedure strengthens the local conditions while preserving
the required sparsity and losing an arbitrarily small amount in
the packing ratio; see \Cref{lemma:seed-increment}.
The resulting seed will be suitable for the next iteration. Starting from a single diamond, finitely many
iterations therefore produce the required seed, with arbitrarily
small total loss in the initial ratio.
The packing theorem then gives the lower bound.

\paragraph{Upper bounds: counting $t$-shadows and extra $t$-sets.}
Before the counting argument, we delete some edges so that the graph forbids more structures. To be more specific, starting with an $n$-vertex $((r-t)k+t,k)$-free $r$-graph $H$,
we delete $O(n^{t-1})$ edges to impose additional local
sparsity conditions; see \Cref{thm:f-to-g-cleanup}.

\medskip
\noindent\emph{When $k$ is odd.}
We group edges into components by joining two edges when they
share at least $t$ vertices. After the deletions, each component
has fewer than $k$ edges and can be built by adding one edge at
a time, with exactly $t$ old vertices, all contained in an earlier
edge. Each addition therefore contributes $\binom rt-1$ new $t$-sets.
A component with $b$ edges contains $b(\binom rt-1)+1$ different $t$-shadows, i.e. $t$-sets that are contained in some edge. However, this count ignores $t$-sets contained in no edge of the hypergraph. Thus, we shall take the non-shadow $t$-sets into account. The main difficulty is that these $t$-sets may be shared by different components, and they might be shadows of the other components, so we analyze the incidence between components and non-shadows.

For example, in the case where $(r,t)=(3,2)$, we recover these missing sets by considering \emph{crossing sets} like $xy$ in the a diamond $\{abx,aby\}$:
they lie across two overlapping edges but in no edge of their
component. They may still appear in other components, so simply
adding the local counts would overcount. We consider an auxiliary incidence graph records all these uses. A cycle in the auxiliary graph would join the components. Since two neighboring components meet at a crossing set, the union of these components will be tight. This would either violates the $((r-t)k+t,k)$-freeness, or the freeness of the structure we forbid from the deletion at the beginning.
The main issue is that this is not always true if we select all the crossing sets in every component. Instead, we select the edges from these components carefully, picking only $b-1$ sets for a component of size $b$. Counting the number of edges and vertices in the auxiliary graph gives the upper bound.

\medskip
\noindent\emph{Why the upper bound improves for even $k$.}
For even $k$ and $r\ge4$, we obtain more missing $t$-sets. We will show that each diamond is associated with two distinct non-shadows.
For each diamond $D$, we may simply pick two non-shadows from the vertices spanned by $D$ if we can. Otherwise, there are many non-shadows of this diamond that are covered by some other edges, so we can extend this piece by adding some of these edges. Note that this is different from the idea of picking components in the case when $k$ is odd. By keep extending the piece, we associate a subgraph $\Psi_D$ and two non-shadows $T_{1,D},T_{2,D}$ to $D$. Now, similar as before, we may consider the incidence auxiliary graph between $\Psi_D$ and non-shadows. With a technical proof, we can show that it is acyclic, and counting the number of edges and vertices of it again gives the upper bound. We remark that the actual process for picking $\Psi_D$ and $T_{1,D},T_{2,D}$ in \cref{subsec:upper-even-k} will be slightly different from the outline above to make the writeup easier.

\medskip

\noindent\textbf{Notation.} Unless otherwise specified, whenever integers $r>t\ge2$ and $k\ge2$
are fixed, we put $d=r-t$.
For an $r$-graph $H$, we write $V(H)$ for the set of vertices spanned
by $H$ and $E(H)$ for its edge set. 
We write
$v(H)=\abs{V(H)}$ and $e(H)=\abs{E(H)}$, and identify $H$ with
$E(H)$ when convenient; thus $|H|$ also denotes its number of edges.
For integers $s$ and $\ell\ge2$, we say that $H$ is
\emph{$(s,\ell)$-free} if it contains no $\ell$ distinct edges
spanning at most $s$ vertices.
We say that two edges $e_1,e_2\in E(H)$ form a \emph{diamond} if
$\abs{e_1\cap e_2}=t$.

\medskip

\noindent\textbf{Organization of the paper.}
The upper bounds in \Cref{thm:BES-r-k-t-odd}
are proved in \Cref{sec:upperbound}.
The corresponding lower bound in \Cref{thm:BES-r-k-t-odd}, together with
\Cref{thm:BES-3-2-even}, are proved in \Cref{sec:lowerbound}.
We prove \Cref{thm:generalized-ramsey} in \Cref{sec:ramsey} and conclude with the remaining triple-system questions in \Cref{sec:conclusion}.

\paragraph{AI usage.} The lower bounds were obtained entirely by the authors without AI assistance.  The authors used ChatGPT‑5.6 as a collaborative tool in establishing the upper bounds and preparing the manuscript. The authors independently verified all mathematical arguments and take full responsibility for the paper’s content.

\section{The upper bound}\label{sec:upperbound}
Throughout this section, let $r>t\ge2$, put $d=r-t$, and fix $k\ge2$.
We will first apply a cleanup by deleting a negligible number of edges to exclude certain small
configurations. Unlike the deletion arguments in \cite{delcourt_postle_limit_2024,shangguan_degenerate_2023}, we only delete edges to destroy the subgraphs in \Cref{def:k-rootable-block}.

\begin{definition}
        Let $B$ be a subgraph of $H$ with $b$ edges. We say that a positive integer $\gamma\leq b$ is \emph{realizable} in $B$ if there exists a subgraph $B_\gamma$ with $\gamma$ edges such that $v(B_\gamma)\leq d\gamma+t$. In this case, we say that $B_\gamma$ is a subgraph realizing $\gamma$.
    \end{definition}
\begin{definition}[$(k,b)$-block]\label{def:k-rootable-block}
Let $1\le b<k$ and set $s\in\{1,\dots, b\}$ to be the remainder of $k$ divided by $b$, where we use the convention $s=b$ when $b\mid k$.
A nonempty $b$-edge subgraph $B$ is a
\emph{$(k,b)$-block} if 
\begin{enumerate}
    \item  $v(B)\le db+t-1$, and 
    \item $s$ is realizable in $B$.
\end{enumerate}
For any $B_s$ realizing $s$ in $B$ and any $R\in\binom{V(B_s)}{t-1}$, we say that $R$ is a \emph{root} of $B$.
\end{definition}

We define $g^{(r)}\bigl(n;(r-t)k+t,k\bigr)$ to be the maximum number of edges in an $n$-vertex $r$-graph which is $(dk+t,k)$-free and contains no $(k,b)$-block for any $b< k$.

\begin{theorem}
    \label{thm:f-to-g-cleanup}
For all $r>t\ge2$, $k\ge2$,
$f^{(r)}\bigl(n;(r-t)k+t,k\bigr)
\le g^{(r)}\bigl(n;(r-t)k+t,k\bigr)
 +(k-1)^2\binom{n}{t-1}.$
\end{theorem}

\begin{proof}
Fix any $(dk+t,k)$-free $r$-graph $H$ on $n$ vertices.
Greedily delete all edges of a $(k,b)$-block whenever one remains,
and assign one root to each deleted block. The deleted blocks are
edge-disjoint. Fix a root $R$ and a size $b<k$, and put
$a=\lfloor(k-1)/b\rfloor$ and $s=k-ab\in[b]$.
If there are $a+1$ deleted $(k,b)$-blocks with root $R$, take all
edges from $a$ of them and an $s$-edge subgraph realizing $s$ that contains $R$ from the last one. Their union has $ab+s=k$ edges
and at most $a(db+t-1)+(ds+t)-a(t-1)=dk+t$ vertices, since all the pieces contain $R$. This contradicts
the freeness of $H$.

Thus at most $ab\le k-1$ edges are deleted for each pair $(R,b)$.
There are at most $\binom{n}{t-1}$ roots and $k-1$ possible sizes,
so the total number of deleted edges is at most
$(k-1)^2\binom{n}{t-1}$.
\end{proof}

\subsection{The upper bound for odd $k$}
By \cref{thm:f-to-g-cleanup}, it is sufficient to show the following.

\begin{theorem}\label{thm:BES-r-k-t-odd-upper}
For all $2\le t<r$ and $k\geq 2$,
$g^{(r)}\bigl(n;(r-t)k+t,k\bigr)
\le
\frac{2}{2\binom{r}{t}-1}\binom{n}{t}.$
\end{theorem}

\begin{proof}
Let $H$ be a $(dk+t,k)$-free $r$-graph on $n$ vertices without $(k,b)$-blocks for any $b<k$. Our goal is to prove that $e(H)\leq \frac{2}{2\binom{r}{t}-1}\binom{n}{t}$. We say that two edges in $H$ are $t$-adjacent if they meet in at least $t$ vertices, and we call the connected subgraphs and the connected components under this adjacency $t$-connected subgraphs and $t$-components, respectively. Every $t$-connected $b$-edge subgraph $F$ has $v(F)\le r+(b-1)(r-t)=db+t.$
Together with the $(dk+t,k)$-freeness, every $t$-component has fewer than $k$ edges.
In fact, we must have $v(F)=db+t$ for any $t$-connected $b$-edge subgraph $F$.
Otherwise, $F$ would form a $(k,b)$-block if $v(F)\le db+t-1$. This is because any $s$-edge connected subgraph of $F$ would satisfy the second condition in the definition of a $(k,b)$-block.
However, this contradicts the assumption that $H$ has no $(k,b)$-block. 

It follows that any $t$-component $F$ admits a \emph{tree structure}. That is, we can order the edges in $E(F)$ as $e_1,\ldots,e_b$, so that 
\begin{enumerate}
    \item for each $i=2,\dots,b$, there exists $j<i$ such that $e_i,e_j$ form a diamond, and
    \item $e_i\cap e_j=e_i\cap (e_1\cup\dots\cup e_{i-1})$.
\end{enumerate}

Recall that $\partial_t F:=\left\{S\in \binom{V(F)}{t}\,\middle|\, \exists e\in E(F): S\subseteq e\right\}$ 
is the set of $t$-shadows of $F$. From the tree structure, we can conclude that $\abs{\partial_t F}=b\left(\binom{r}{t}-1\right)+1.$

For each $i=2,\dots,b$, we fix a choice of $j<i$ such that $D_{i,F}:=\{e_i,e_j\}$ is a diamond. We then pick one vertex from $e_i\setminus e_j$, one vertex from $e_j\setminus e_i$, and $t-2$ vertices from $e_i\cap e_j$ to form a set $p_{i,F}$. We call $p_{i,F}$ a \emph{crossing set} of $F$. Note that, because of the tree structure, $p_{i,F}\notin \partial_t F$. Moreover, since $p_{i,F}\subseteq e_1\cup\dots\cup e_i$, we know that the crossing sets $p_{2,F},\dots,p_{b,F}$ are all distinct and the diamonds $D_{2,F},\dots,D_{b,F}$ do not contain the crossing sets associated with other diamonds.

We construct an auxiliary bipartite graph $G$ on the vertex set $\cC\sqcup \cT$, where $\cC$ is the set of $t$-components in $H$ and $\cT=\binom{V(H)}{t}$. We put an edge between $F\in\cC$ and $p\in\cT$ if $p\in \partial_t F$ or $p=p_{i,F}$ for some $i$. We call the first type a \emph{shadow edge} and the second type a \emph{crossing edge}. We claim that $G$ is acyclic.

Suppose not, then there is a cycle
$p_1,F_1,p_2,\dots,p_z,F_z,p_1$
in $G$. Note that $p_i$ can not be in both $\partial_t F_{i-1}$ and $\partial_t F_i$, otherwise the two $t$-components $F_{i-1}$ and $F_i$ would merge. For each of $p_i,p_{i+1}$, choose an edge of $F_i$ containing it
when it is a shadow set, and choose its associated diamond when it
is a crossing set. Let $F'_i$ be an inclusion-minimal $t$-connected
subgraph containing these chosen edges, and put $\lambda_i=e(F'_i)$.
Thus each of $p_i,p_{i+1}$ has a specified supporting subgraph of at
most two edges inside $F'_i$. Since $F'_1,\dots,F'_z$ come from different $t$-components, the edges are disjoint. Therefore,  $B:=\bigcup_{i=1}^z F'_i$ is a subgraph with $b:=\sum_{i=1}^z\lambda_i$ edges. Each $F'_i$ spans $d\lambda_i+t$ vertices. Adding the pieces in
cyclic order, consecutive pieces share a $t$-set, and the last
piece meets the preceding union in $p_z\cup p_1$, of size at least
$t+1$. Hence $v(B)\le db+t-1$.

\begin{claim}\label{claim:k-odd-claim}
    Every positive integer $\gamma\leq b$ is realizable in $B$.
\end{claim}
\begin{poc}
    We may assume without loss of generality that $\lambda_1$ is the maximum among $\lambda_1,\dots,\lambda_z$.
    Not all $\lambda_i$ can equal $1$, since then every endpoint
    would be a shadow set in both adjacent components. Thus
    $\lambda_1\ge2$. If $\lambda_1=2$, at least one of $p_1,p_2$
    must be in $\partial_t F_1$. This is because two distinct crossing sets have different
    supporting diamonds. Reversing the cycle if necessary, assume
    that $p_2$ is a shadow set in this case.

    If $\gamma\le\lambda_1$, take any $t$-connected $\gamma$-edge
    subgraph of $F'_1$. We may therefore assume that
    $\gamma>\lambda_1$. Let $\eta\ge2$ be minimal with
    $\lambda_1+\cdots+\lambda_\eta\ge\gamma$, and put
    $u=\gamma-(\lambda_1+\cdots+\lambda_{\eta-1})$.
    Then $1\le u\le\lambda_\eta$.
    
    If $u\ge2$, grow the chosen supporting edge or diamond for
    $p_\eta$ to a $t$-connected $u$-edge subgraph $F''_\eta$
    of $F'_\eta$. The same is possible when $u=1$ and $p_\eta$
    is a shadow set. In either case,
    $B_\gamma=F'_1\cup\cdots\cup F'_{\eta-1}\cup F''_\eta$
    has $\gamma$ edges . Consecutive pieces share a $t$-set, so
    $v(B_\gamma)\le d\gamma+t$.

    It remains to consider $u=1$ when $p_\eta$ is a crossing set.
    Its supporting diamond has two edges. We can also choose a
    $t$-connected $(\lambda_1-1)$-edge subgraph $F''_1\subseteq F'_1$
    containing $p_2$: its chosen support has at most two edges if
    $\lambda_1\ge3$, and one edge if $\lambda_1=2$.
    Thus, we may take $B_\gamma$ to be the union of $F''_1$, $F'_2\cup\dots\cup F'_{\eta-1}$, and the supporting diamond
    of $p_\eta$ in $F_\eta$ in this case.
\end{poc}

We are ready to get a contradiction by showing that $B$ contains a forbidden structure. If $b\geq k$, then by \cref{claim:k-odd-claim}, we know that there is a subgraph $B_k$ with $k$ edges and $v(B_k)\leq dk+t$, which contradicts the $(dk+t,k)$-freeness. If $b<k$, then \cref{claim:k-odd-claim} together with the fact that $v(B)\leq db+t-1$, we know that $B$ is a $(k,b)$-block, which is also impossible. Thus, the graph $G$ is acyclic.

Let $m$ be the number of $t$-components.
A component with $b$ edges has $b(\binom{r}{t}-1)+1$ shadow edges and $b-1$
crossing edges in $G$. Hence $e(G)=\binom{r}{t}e(H)$. Since $G$ is a forest,
$\binom{r}{t}e(H)\le m+\binom nt.$ On the other hand, the component shadows
are disjoint, so \[\left(\binom{r}{t}-1\right)e(H)+m=|\partial_tH|\le\binom nt.\]
Adding these inequalities gives $\left(2\binom{r}{t}-1\right)e(H)\le2\binom nt,$ as required.
\end{proof}

\subsection{The upper bound for even $k$}\label{subsec:upper-even-k}
By \cref{thm:f-to-g-cleanup}, it is sufficient to show the following.

\begin{theorem}\label{thm:BES-r-k-t-even-upper}
For all $r\geq 4$, $2\le t<r$ and even $k\geq 2$,
$g^{(r)}\bigl(n;(r-t)k+t,k\bigr)
\le
\frac{1}{\binom{r}{t}}\binom{n}{t}.$
\end{theorem}
\begin{proof}
    Let $H$ be a $(dk+t,k)$-free $r$-graph on $n$ vertices without $(k,b)$-blocks for any $b<k$. Similar to the proof of \cref{thm:BES-r-k-t-odd-upper}, if $F_1,\dots,F_m$ are the $t$-components of $H$, then $\abs{\partial_t H}=\abs{\bigcup_{i=1}^m\partial_t F_i}= e(H)\left(\binom{r}{t}-1\right)+m$. We again define $\cT=\binom{V(H)}{t}, \cT_1=\partial_t H, \cT_0=\cT\setminus \cT_1$. Let $\cD$ be the family of all diamonds in $H$. It was also proved in the proof of \cref{thm:BES-r-k-t-odd-upper} that $\abs{\cD}\geq e(H)-m$. Thus, it is sufficient to show that $\abs{\cD}\leq \abs{\cT_0}$. Once this is proved, \[\binom{n}{t}\geq \abs{\cT_1}+\abs{\cD}\geq e(H)\left(\binom{r}{t}-1\right)+m+e(H)-m=e(H)\binom{r}{t},\]
    and the theorem follows.

    If $\cD=\varnothing$, there is nothing to prove. Hence we may assume
    that $\cD\ne\varnothing$, and in particular, $k\ge4$.

    \paragraph{Associating two missing sets with each diamond.}
    For each diamond $D\in\cD$, we fix an ordering of its edges $D=\{e_1,e_2\}$, and consider the following process. We maintain two lists of distinct edges $e_1,\dots,e_i$ and $f_3,\dots,f_i$. For $i\geq 2$, suppose that we have already defined $e_1,\dots,e_i$, then we say that a $t$-set $T\subseteq e_1\cup\dots\cup e_{i}$ is a candidate if $T\cap (e_i\setminus (e_1\cup\dots\cup e_{i-1}))\neq \varnothing$ and $T\cap ((e_1\cup\dots\cup e_{i-1})\setminus e_i)\neq \varnothing$. If there are at least two distinct candidates contained in distinct edges
    outside $\{e_1,\ldots,e_i,f_3,\ldots,f_i\}$, choose two of them    and name the corresponding edges $e_{i+1}$ and $f_{i+1}$.
    The process terminates if there are no such two edges. Suppose the process terminates with $\ell_D$ edges in the first list, we define $I_D=\{e_1,\dots,e_{\ell_D}\}$, $U_D=\{f_3,\dots,f_{\ell_D}\}$, $\Psi_D=I_D\cup U_D$.
    \begin{claim}
        We have $\abs{\Psi_D}< k$.
    \end{claim}
    \begin{poc}
       In the order $e_1,e_2,e_3,f_3,e_4,f_4,\ldots$, every edge
        after the first meets the union of the preceding edges in
        at least $t$ vertices. If the process reached $k$ edges,
        its first $k$ edges would therefore span at most $dk+t$
        vertices, a contradiction.
    \end{poc}

    \begin{claim}\label{claim:even-tight}
        For $2\le i\le\ell_D$ and distinct indices $3\le j_1<\cdots<j_a\le\min\{i+1,\ell_D\}$, allowing $a=0$, the graph $B$ formed by the edges $e_1,\dots,e_i,f_{j_1},\dots,f_{j_a}$ satisfies $v(B)=d(i+a)+t$.
    \end{claim}
    \begin{poc}
        From the construction, we have $v(B)\leq d(i+a)+t$. If the
    inequality were strict, put $b=i+a<k$ and let
    $s\in[b]$ be the remainder of $k$ modulo $b$, with $s=b$ when
    $b\mid k$. The first $s$ edges
    in the ordering $e_1,\dots,e_i,f_{j_1},\dots,f_{j_a}$ span at most $ds+t$ vertices.
    Hence $B$ would be a $(k,b)$-block, a contradiction.
    \end{poc}

    We also record a consequence for an edge $h$ not yet chosen.
    Write $S_i=e_1\cup\cdots\cup e_i$. If $h$ contains a candidate,
    then $|h\cap S_i|=t$. Indeed, a larger intersection would make
    $\{e_1,\ldots,e_i,h\}$ a $(k,i+1)$-block: all its prefix sizes
    are realizable, and $i+1\le\ell_D+1<k$.
    In particular, such an edge contains at most one candidate.

    From the claim, we know that
        $v(e_1\cup\dots\cup e_i)=di+t.$
    Therefore, we know that both $e_i\setminus S_{i-1}$ and $ S_{i-1}\setminus e_i$ contain at least $d$ vertices, and $e_i\cap S_{i-1}$ contains $t$ vertices. Thus, there are at least $d^2\binom{t}{t-2}$ candidates since we can pick one vertex each from $e_i\setminus S_{i-1}$ and $ S_{i-1}\setminus e_i$, and pick $t-2$ vertices from $e_i\cap S_{i-1}$ to form a candidate. From the assumption of the parameters, there are at least $3$ candidates. Therefore, in the $\ell_D$-th round where the process terminates, there are at least two candidates $T_{1,D},T_{2,D}$ that are not contained in any new edges. Neither $T_{1,D}$ nor $T_{2,D}$ lies in a chosen edge: it meets the new vertices of $e_{\ell_D}$ and also $S_{\ell_D-1}\setminus e_{\ell_D}$,
    while each $f_j$ meets $S_{\ell_D}$ only in its attaching
    $t$-set in $S_{j-1}$. Thus both sets belong to $\cT_0$.

    Now, for each diamond $D\in\cD$, we fix a choice of $I_D,U_D,\Psi_D,T_{1,D},T_{2,D}$. If $\cD$ is non-empty, then we construct an auxiliary (multi-)graph $\Tilde{G}$ as follows. The vertex set of $\Tilde{G}$ is the set $\cT_0$, and we put an edge between $T_{1,D},T_{2,D}$ for every diamond $D$ in $H$. We claim that the graph $\Tilde{G}$ is indeed a simple acyclic graph. 

    By \cref{claim:even-tight}, $v(\Psi_D)=d(2\ell_D-2)+t$.
    In the order $e_1,\ldots,e_{\ell_D},f_3,\ldots,f_{\ell_D}$,
    each edge after the first adds exactly $d$ vertices.
    In particular, each $f_i$ meets $S_{\ell_D}$ exactly in its
    attaching $t$-set in $S_{i-1}$, and the sets
    $f_i\setminus S_{\ell_D}$ are pairwise disjoint.
    The attaching sets for $e_i,f_i$ are distinct and each meets
    both $e_{i-1}\setminus S_{i-2}$ and $S_{i-2}\setminus e_{i-1}$.
    Thus neither is contained in an earlier $e$-edge or an earlier
    attaching set. It follows that no two chosen edges meet in
    $t$ vertices except $e_1,e_2$: $D$ is the only diamond in $\Psi_D$.
    We record two further properties below.

    \begin{claim}\label{claim:even-local-properties}
        For every $D\in\cD$, the following statements hold.
        \begin{enumerate}
            \item Every nonempty $F\subseteq\Psi_D$ satisfies $v(F)\geq d|F|+t$.
            If $|F|\geq2$ and equality holds, then $D\subseteq F$.

            \item Every integer $1 \leq \gamma\leq 2\ell_D-2$ is realizable in $\Psi_D$. Moreover, if $\gamma\geq \ell_D$, then we can pick the realizing graph $B_\gamma$ such that $I_D\subseteq B_\gamma$, and hence $T_{1,D}\cup T_{2,D}\subseteq V(B_\gamma)$.
        \end{enumerate}
    \end{claim}

    \begin{poc}
        For the first part, order the edges in $F$ as a subsequence of 
        $e_1,\dots,e_{\ell_D},f_3,\dots,f_{\ell_D}.$
        The first selected edge contributes $r=d+t$ vertices. Every later
        selected edge has $d$ vertices which were new at the time
        that edge appeared, and hence these vertices do not belong to
        any earlier selected edge. Therefore
           $ v(F)\geq r+d(|F|-1)=d|F|+t.$
        If equality holds and $|F|\geq2$, then the first two selected
        edges must meet in exactly $t$ vertices. By the uniqueness of
        the diamond in $\Psi_D$, these two edges must form $D$.

        For the second part of the claim, we may simply take $B_\gamma$ to be the first $\gamma$ edges in the ordering
        $e_1,\dots,e_{\ell_D},f_3,\dots,f_{\ell_D}.$
        It is not hard to check that $B_\gamma$ satisfies the conditions.
    \end{poc}

    \paragraph{Controlling overlaps of the associated subgraphs.}
    We next show that two such subgraphs share at most one edge,
    and that sharing an edge makes every intermediate size realizable.

    \begin{claim}\label{claim:two-gadget-overlap}
        For any two distinct diamonds $D,D'$, set $B:=\Psi_D\cup\Psi_{D'}$ and $b=e(B)$.
        We have
            $|E(\Psi_D)\cap E(\Psi_{D'})|\leq1.$
        Moreover, if the equality holds, then every positive integer $\gamma\leq b$ is realizable in $B$.
    \end{claim}

    \begin{poc}
Write $\Psi_1=\Psi_D$, $\Psi_2=\Psi_{D'}$, and
$\ell_1=\ell_D\ge\ell_2=\ell_{D'}$, relabelling if necessary.
Suppose that $c=|E(\Psi_1)\cap E(\Psi_2)|\ge1$.
By \cref{claim:even-local-properties}, every size at most
$2\ell_1-2$ is realizable in $\Psi_1$.

Give each edge $e_i,f_i$ its index $i$, and choose a shared edge
$h$ whose index $i_1$ in $\Psi_1$ is smallest. Let $i_2$ be its
index in $\Psi_2$. For $a\in\{1,2\}$, let $B_a$ consist of $h$
and the first $i_a-1$ $e$-edges of $\Psi_a$.
Then $B_1\cap B_2=\{h\}$ by the choice of $h$, and
$v(B_a)\le di_a+t$ by the construction. Thus $B_1\cup B_2$
realizes $b_0=i_1+i_2-1\le\ell_1+\ell_2-1\le2\ell_1-1$.
Append the remaining $e$-edges of each copy in order, followed
by the remaining $f$-edges, omitting repetitions. Each added edge
meets the preceding union in at least $t$ vertices, so every size
from $b_0$ to $b$ is realizable. Together with the sizes realized
in $\Psi_1$, this covers all of $[b]$. In particular, $b<k$.

If $c\ge2$, \cref{claim:even-local-properties} gives
$v(\Psi_1\cap\Psi_2)\ge dc+t+1$: equality with $dc+t$ would
force the intersection to contain both distinct diamonds, whereas
each $\Psi_a$ contains only its initial diamond. Consequently,
$    v(B)\le v(\Psi_1)+v(\Psi_2)-v(\Psi_1\cap\Psi_2)
    \le db+t-1.$
Since every size in $[b]$ is realizable, $B$ would be a
$(k,b)$-block, a contradiction. Hence $c=1$, as required.
\end{poc}

\begin{claim}\label{claim:disjoint-path-realizable}
        Suppose that
       $T_1,D_1,T_2,D_2,\dots,T_z,D_z,T_{z+1}$
        is a walk in $\Tilde{G}$ such that $D_1,\dots,D_z$ are all
        distinct. Assume that $\Psi_{D_1},\dots,\Psi_{D_z}$ are mutually
        edge-disjoint and set $b=\sum_{i=1}^ze(\Psi_{D_i})$. Then every positive even $\gamma\leq b$ is realizable in $\bigcup_{i=1}^z \Psi_{D_i}$.
    \end{claim}

    \begin{poc}
        Put $\ell_i=\ell_{D_i}$.
        If $\gamma\le 2\ell_i-2$ for some $i$, use
        \cref{claim:even-local-properties}. Otherwise, since 
        $\sum_{i=1}^z(2\ell_i-2)=b\geq\gamma$,
        we may choose an inclusion-minimal interval
        $J\subseteq[z]$ such that
        $\sum_{i\in J}(2\ell_i-2)\geq\gamma$.
        Since both sides are even, we may write $\sum_{i\in J}(2\ell_i-2)=\gamma+2x$
        for some nonnegative integer $x$. In the case
        $\gamma>2\ell_i-2$ for every $i$, the interval $J$ contains
        at least two indices. 
        By the minimality of $J$, deleting either endpoint of $J$
        makes the sum smaller than $\gamma$. Hence, the two
        endpoint terms $2\ell_i-2$ of $J$ are both at least $2x+2$, while every other
        term is at least $2$. Thus, 
        \[
            \gamma+2x
            \geq 2(2x+2)+2(|J|-2)
            =4x+2|J|,
        \]
        which implies that                      $  \sum_{i\in J}\ell_i
            =\frac{\gamma}{2}+x+|J|\le\gamma.$
        Therefore, we may choose integers $\ell_i\le\gamma_i\le 2\ell_i-2$ for $i\in J$
        with $\sum_{i\in J}\gamma_i=\gamma$.
        By \cref{claim:even-local-properties}, there are subgraphs
        $B_i\subseteq\Psi_{D_i}$ with $\gamma_i$ edges containing
        $I_{D_i}$ and satisfying $v(B_i)\le d\gamma_i+t$.
        Their edge sets are disjoint, and consecutive subgraphs
        contain the common $t$-set $T_{i+1}$. Therefore their union
        has $\gamma$ edges and at most
        $\sum_{i\in J}(d\gamma_i+t)-(|J|-1)t=d\gamma+t$ vertices.
    \end{poc}

    Now, we want to show that the assumption that $E(\Psi_{D_1}),\dots,E(\Psi_{D_z})$ are mutually disjoint in \cref{claim:disjoint-path-realizable} always holds. The next two claims handle the only possible overlap on a shortest counterexample that exactly one edge is shared by the first and last subgraphs.
    \begin{claim}\label{claim:even-realizable}
        Suppose that
            $T_1,D_1,T_2,D_2,\dots,T_z,D_z,T_{z+1}$
        is a walk in $\Tilde{G}$ such that $D_1,\dots,D_z$ are all
        distinct. Assume that $E(\Psi_{D_1}),\dots,E(\Psi_{D_z})$ are mutually disjoint, except that $\abs{E(\Psi_{D_1})\cap E(\Psi_{D_z})}=1$. Set $\ell_i=\ell_{D_i}$ and $b=\abs{\bigcup_{i=1}^z E(\Psi_{D_i})}=\sum_{i=1}^z (2\ell_i-2)-1$. Then every $\gamma\in [b]$ such that at least one of the following holds is realizable in $\bigcup_{i=1}^z\Psi_{D_i}$.
        \begin{enumerate}
            \item $\gamma\leq 2\ell_i-2$ for some $i\in [z]$,
            \item $\sum_{i\in J}\ell_i\leq \gamma\leq \sum_{i\in J}(2\ell_i-2)$ for some non-empty interval $J\subseteq [z]$,
            \item $\gamma\leq 2\ell_1+2\ell_z-5$, or
            \item $\gamma=\sum_{i\notin J}(2\ell_i-2)-1$ for some (possibly empty) interval $J\subseteq [2,z-1]$.
        \end{enumerate}
    \end{claim}
    \begin{poc}
        Let $E(\Psi_{D_1})\cap E(\Psi_{D_z})=\{h\}$. We prove the four cases separately. 

        For the first case where $\gamma\leq2\ell_i-2$ for some $i\in[z]$. By the second part of
        \cref{claim:even-local-properties}, $\gamma$ is realizable
        in $\Psi_{D_i}$, and hence it is realizable in $\bigcup_{i=1}^z\Psi_{D_i}$.

        Next, suppose that there is a nonempty interval
        $J\subseteq[z]$ such that
        $            \sum_{i\in J}\ell_i
            \leq\gamma
            \leq
            \sum_{i\in J}(2\ell_i-2).$
        We may choose integers $\gamma_i$, for $i\in J$,
        such that $\ell_i\leq\gamma_i\leq2\ell_i-2$ and $\sum_{i\in J}\gamma_i=\gamma$.
        By the second part of
        \cref{claim:even-local-properties}, for every $i\in J$ there
        is a $\gamma_i$-edge subgraph $B_i\subseteq\Psi_{D_i}$ such that $I_{D_i}\subseteq B_i$ and $v(B_i)\leq d\gamma_i+t$.
        In particular, $B_i$ contains both $T_i$ and $T_{i+1}$.

        Suppose first that $J\neq[z]$. Then the graphs
        $\Psi_{D_i}$, for $i\in J$, are mutually edge-disjoint.
        Therefore, we may set $B_\gamma$ to be $\bigcup_{i\in J}B_i$, which has exactly $\gamma$ edges. Adding these graphs in their
        order along the interval $J$, we obtain $v(B_\gamma)
            \leq
            \sum_{i\in J}(d\gamma_i+t)-(|J|-1)t=d\gamma+t.$

        It remains to consider $J=[z]$. If $h$ is not contained in
        both $B_1$ and $B_z$, the same argument applies. Suppose that
        $h\in B_1\cap B_z$. The graph $\bigcup_{i=1}^z B_i$
        then has $\gamma-1$ distinct edges. 
        Since $T_z\notin \partial_t(H)$ and $h\in E(H)$, we know that $T_z$ is not contained in $h$. Therefore, we have $\abs{h\cup T_z}\geq r+1.$
        Thus, 
        \[
            v\left(\bigcup_{i=1}^z B_i\right)
            \le\sum_{i=1}^z(d\gamma_i+t)-(z-2)t-(r+1)
            =d(\gamma-1)+t-1.
        \]
        Moreover, all $I_{D_i}$ are included in $\bigcup_{i=1}^z B_i$, so we may add a new edge from some $U_{D_i}$ to it to form $B_\gamma$. In this way, the graph $B_\gamma$ has $\gamma$ edges and $v(B_{\gamma})\leq v(\bigcup_{i=1}^z B_i)+d\leq d\gamma+t-1.$ 

        Next, suppose that $\gamma\leq 2\ell_1+2\ell_z-5$. Since
        $e(\Psi_{D_1}\cup\Psi_{D_z})
=(2\ell_1-2)+(2\ell_z-2)-1=2\ell_1+2\ell_z-5,$
        we may apply \cref{claim:two-gadget-overlap} to
        $\Psi_{D_1}\cup\Psi_{D_z}$. This shows that $\gamma$ is
        realizable.

        Finally, suppose that
        $\gamma
            =
            \sum_{i\notin J}(2\ell_i-2)-1$
        for some possibly empty interval
        $J\subseteq[2,z-1]$. If $J$ is empty, then $\gamma=b$. In this case, we take $B_\gamma$ to be $\bigcup_{i=1}^z\Psi_{D_i}$. Similar to the calculation in the second case, we have
        \begin{equation}\label{eq:even-union-size}
            v(\bigcup_{i=1}^z\Psi_{D_i})\leq db+t-1,
        \end{equation}
        which is stronger than the bound $db+t$ we need.

        Suppose that $J$ is nonempty, and write $J=[p,q]$.
        Let $B_\gamma:=
            \bigcup_{i<p}\Psi_{D_i}
            \ \cup\
            \bigcup_{i>q}\Psi_{D_i}.$
        The only repeated edge in this union is $h$, so
       $e(B_\gamma)
            =
            \sum_{i\notin J}(2\ell_i-2)-1
            =\gamma.$
        We know that 
    $v(\bigcup_{i<p}\Psi_{D_i})\leq d\sum_{i<p}(2\ell_i-2)+t$ and $v(\bigcup_{i>q}\Psi_{D_i})\leq d\sum_{i>q}(2\ell_i-2)+t.$
        Since both of them contain the edge $h$, we have
        $v(B_\gamma)\leq \left(d\sum_{i<p}(2\ell_i-2)+t\right)+\left(d\sum_{i>q}(2\ell_i-2)+t\right)-r=d\gamma+t.$
    \end{poc}
    Indeed, the four conditions in \cref{claim:even-realizable} cover all $\gamma\leq b$. We will show this in the following claim.
    \begin{claim}\label{claim:realizable-cover}
        For integers $\ell_1,\ldots,\ell_z\ge2$, every positive
        $\gamma\le\sum_{i=1}^z(2\ell_i-2)-1$ satisfies at least
        one of the four numerical conditions in
        \cref{claim:even-realizable}.
    \end{claim}
    \begin{poc}
        Write $a_i=2\ell_i-2$. We may assume that
        $\gamma>\max_i a_i$ and $\gamma\ge a_1+a_z$, since otherwise
        the first or third condition holds.
        Call an interval $J$ minimal if its sum is at least $\gamma$
        but neither endpoint can be removed. Write
        $\sum_{i\in J}a_i=\gamma+x$. Such an interval has at least
        two terms, its endpoints are at least $x+1$, and its other
        terms are at least $2$. Therefore
        \begin{equation}\label{eq:realizable-cover-interval}
            \gamma+x\ge2x+2|J|-2.
        \end{equation}
        If the inequality is strict, both sides are even, so
        $\gamma+x\ge2x+2|J|$. It follows that
        $\sum_{i\in J}\ell_i=(\gamma+x)/2+|J|\le\gamma$,
        giving the second condition.

        We may therefore assume equality for every minimal interval.
        Its terms then have the form $x+1,2,\ldots,2,x+1$;
        in particular, $x$ and $\gamma$ are odd. Choose a minimal
        interval $[u,v]$ with $v$ maximal. Removing its first term
        leaves sum $\gamma-1$. If $v<z$, append $a_{v+1}$ and then
        remove as many leftmost terms as possible while keeping
        sum at least $\gamma$. This gives another minimal interval
        ending at $v+1$, a contradiction. Hence $v=z$.

        We now have $a_{u+1}=\cdots=a_{z-1}=2$ and
        $\sum_{i=u+1}^z a_i=\gamma-1$.
        The even integer $\gamma+1-a_1$ is strictly greater than
        $a_z$ and at most $\gamma-1$. Thus some suffix $[y,z]$,
        with $u<y<z$, has sum $\gamma+1-a_1$.
        Taking $J_0=[2,y-1]$, possibly empty, gives
        $\sum_{i\notin J_0}a_i-1=a_1+\sum_{i=y}^z a_i-1=\gamma$,
        which is the fourth condition.
    \end{poc}
    \begin{cor}\label{cor:path-gadgets-disjoint}
        Suppose that
            $T_1,D_1,T_2,D_2,\dots,T_z,D_z,T_{z+1}$
        is a walk in $\Tilde{G}$ such that $D_1,\dots,D_z$ are all
        distinct. Then $E(\Psi_{D_1}),\dots,E(\Psi_{D_z})$ are mutually disjoint.
    \end{cor}
    \begin{proof}
        Suppose the edge sets are not mutually disjoint, then we may assume without loss of generality that only $E(\Psi_{D_1})$ and $E(\Psi_{D_z})$ intersect by shortening the walk. From \cref{claim:two-gadget-overlap}, we know that they intersect in exactly one edge. Thus, we may apply \cref{claim:even-realizable} and \cref{claim:realizable-cover} to conclude that every positive integer $\gamma\leq b=\abs{\bigcup_{i=1}^z E(\Psi_{D_i})}$ is realizable in $\bigcup_{i=1}^z E(\Psi_{D_i})$. If $b\geq k$, then $k$ will be realizable, which contradicts the $(dk+t,k)$-freeness of $H$. Otherwise, from \cref{eq:even-union-size} in the proof of \cref{claim:even-realizable}, we have
        $v(\bigcup_{i=1}^z\Psi_{D_i})\leq db+t-1.$
        Thus, $\bigcup_{i=1}^z\Psi_{D_i}$ forms a $(k,b)$-block since the remainder $s$ of $k$ modulo $b$ is realizable. Again, this yields a contradiction.
    \end{proof}

    \paragraph{Completing the forest argument.}
    The overlap analysis now lets us count along an auxiliary cycle
    without counting any hyperedge twice. The evenness of $k$
    supplies the required remainder size.

    \begin{claim}\label{claim:even-auxiliary-forest}
        The graph $\Tilde{G}$ is a simple acyclic graph.
    \end{claim}

    \begin{poc}

        Suppose that $\Tilde{G}$ contains a cycle or a pair of parallel
        edges, then we may find a walk with the same starting and ending vertex
            $T_1,D_1,T_2,D_2,\dots,T_z,D_z,T_1$
        such that $D_1,\dots,D_z$ are all distinct. 
        \cref{cor:path-gadgets-disjoint} shows that $\Psi_{D_1},\dots,\Psi_{D_z}$ are mutually edge-disjoint.

        Put $B=\bigcup_{i=1}^z\Psi_{D_i}$ and
        $b=e(B)=\sum_{i=1}^ze(\Psi_{D_i})$, which is even.
        Adding the subgraphs in cyclic order, consecutive pieces
        share a $t$-set, and the last meets the preceding union in
        the two distinct sets $T_z,T_1$. Thus $v(B)\le db+t-1$.

        By \cref{claim:disjoint-path-realizable}, every even
        $\gamma\leq b$ is realized by a $\gamma$-edge subgraph of
        $B$ spanning at most $d\gamma+t$ vertices.
        If $b\geq k$, take $\gamma=k$, contradicting the
        $(dk+t,k)$-freeness of $H$. Suppose instead that $b<k$, and let
        $s\in[b]$ be the remainder of $k$ modulo $b$. Since both
        $k$ and $b$ are even, $s$ is even. Hence, $s$ is realizable in $B$. Therefore, $B$ is a
        $(k,b)$-block, a contradiction.
    \end{poc}

    Since $\Tilde{G}$ is a forest,
    $|\cD|=e(\Tilde{G})\le v(\Tilde{G})=|\cT_0|$.
    Together with the shadow count at the start of the proof,
    this gives $e(H)\binom rt\le\binom nt$, as required.
\end{proof}

\section{The lower bound}\label{sec:lowerbound}
Throughout this section, whenever integers $r>t$ are fixed, we write
$d=r-t$ and define
$\xi^{(r)}(n;dk+t,k)$ to be the maximum number of edges in an
$n$-vertex $r$-graph that is $(dk+t,k)$-free and
$(d\ell+t-1,\ell)$-free for every integer
$2\le\ell\le k-1$.
Clearly,
$\xi^{(r)}(n;dk+t,k)\le f^{(r)}(n;dk+t,k)$.
The high-girth packing theorem of Delcourt and
Postle~\cite[Theorem~1.3]{delcourt2022finding} gives, for all fixed
$r>t\ge2$ and $k\ge2$, a partial Steiner system with at least
$(1-o(1))\binom nt/\binom rt$ edges and no $(di+t,i)$-configuration
for any $2\le i\le k$. In particular,
\begin{equation}\label{eq:xi-packing-lower}
    \xi^{(r)}(n;dk+t,k)
    \ge\left(\frac{1}{t!\binom rt}-o(1)\right)n^t.
\end{equation}
This proves the lower bound for even $k$ in
\Cref{thm:BES-r-k-t-odd}; the stronger bound for $\xi^{(r)}$ will
also be used in \Cref{sec:ramsey}.

We shall focus on the lower bound for odd $k$ in
\Cref{thm:BES-r-k-t-odd} and the strict improvement asserted in
\Cref{thm:BES-3-2-even}.
Our proofs use the conflict-free hypergraph matching method developed
independently by Delcourt and Postle
\cite{delcourt2022finding}
and by Glock, Joos, Kim, K\"uhn, and Lichev
\cite{glock_conflict_2024}.
More precisely, we use a consequence of this method established in
\cite[Theorem~3.1]{glock_64_2024}.
Before stating it in the form needed here, we introduce the following
notions.

\begin{definition}\label{def:P-configuration}
Fix integers $r>t\ge2$, and let $H$ be an $r$-uniform hypergraph.
A nonempty subgraph $F$ is \emph{tight} if $v(F)\le d|F|+t$.
For $p\in\binom{V(H)}{t}$, let $\mu_p(H;r,t)$ be the minimum
number of edges in a tight subgraph whose vertex set contains $p$:
\[
    \mu_p(H;r,t)
    :=
    \min\Bigl\{
        |H_0|:
        H_0\subseteq H,\ 
        p\subseteq V(H_0),\
        v(H_0)\le d|H_0|+t
    \Bigr\}.
\]
If no such subhypergraph exists, we set $\mu_p(H;r,t)=\infty$.
When $r$ and $t$ are clear from the context, we simply write
$\mu_p(H)$.
\end{definition}

\begin{definition}[$(r,t)$-seed]\label{def:seed}
Let $V$ be a finite set.
A triple $\mathcal{T}=(V,E,J)$ is called an \emph{$(r,t)$-seed} if
\[
    \emptyset\neq E\subseteq\binom{V}{r},
    \qquad
    V(E)=V,
    \qquad
    J\subseteq\binom{V}{t},
    \qquad
    \partial_t(E)\subseteq J.
\]
Here $\partial_t(E)=\bigcup_{e\in E}\binom{e}{t}$ is the $t$-shadow
of $E$.
The \emph{packing ratio} of $\mathcal{T}$ is
$\gamma(\mathcal{T})=\frac{|E|}{|J|}$.
For every $p\in\binom{V}{t}$, define
$\mu_p(\mathcal{T})=\mu_p(E;r,t)$.
Let $R = \binom{V}{t}\setminus J$. 
Define the girth of $\mathcal{T}$ to be $\min\{\mu_p(\mathcal{T}): p\in R\}$.
If $R=\varnothing$, we define the girth of $\mathcal T$ to be
$\infty$.
\end{definition}

\begin{definition}[$(\alpha,\beta)$-friendly]\label{def:friendly}
Let $\alpha,\beta\ge1$ be integers.
An $(r,t)$-seed $\mathcal{T}=(V,E,J)$ is called
\emph{$(\alpha,\beta)$-friendly} if the following conditions hold.
\begin{enumerate}
    \item
    For every integer $2\le \ell\le\alpha$, the $r$-graph $E$
    is $(d\ell+t-1,\ell)$-free.

    \item
    For every odd integer $3\le \ell\le\alpha$, the $r$-graph $E$
    is $(d\ell+t,\ell)$-free.

    \item
    For every $p\in\binom{V}{t}\setminus J$, we have
    $\mu_p(\mathcal T)>\beta$.
\end{enumerate}
\end{definition}

The following lemma records a slightly strengthened consequence of
\cite[Theorem~3.1]{glock_64_2024} in our notation.\footnote{The
strengthening is implicit in the proof of
\cite[Theorem~3.1]{glock_64_2024}. In Step~2 of that proof, the
conflict family $\mathcal C_1$ excludes every
$(\ell(r-t)+t,\ell)$-configuration using edges from at least two
packed copies, for every $2\le\ell\le k$. Step~6 produces a
$\mathcal C$-free matching. Hence an
$(\ell(r-t)+t-1,\ell)$-configuration in the resulting hypergraph
would either lie in one copy, contradicting the hypothesis on the
seed, or use at least two copies, contradicting the
$\mathcal C$-free construction.}

\begin{lemma}\label{lemma:glock-packing}
Let \(r>t\ge2\), \(k\ge2\), and \(d=r-t\).
Let \(\mathcal T=(V,E,J)\) be an \((r,t)\)-seed such that
the \(r\)-graph \((V,E)\) is \((dk+t,k)\)-free and
$(d\ell+t-1,\ell)$-free
for every $2\le \ell\le k-1$. Suppose further that the girth of $\mathcal{T}$ is greater
than \(\frac{k}{2}\). Then
$\liminf_{n\to\infty}
    \frac{f^{(r)}(n;dk+t,k)}{n^t}
    \ge
    \liminf_{n\to\infty}
    \frac{\xi^{(r)}(n;dk+t,k)}{n^t}
    \ge
    \frac{|E|}{t!\,|J|}
    =
    \frac{\gamma(\mathcal T)}{t!}.$
\end{lemma}

We can directly derive the following lemma from Lemma~\ref{lemma:glock-packing}.

\begin{lemma}
\label{lemma:friendly-seed-packing}
Let $r>t\ge2$, let $k\ge3$ be odd, and put $d=r-t$.
If $\mathcal{T}=(V,E,J)$ is a $(k,k)$-friendly $(r,t)$-seed, then
    $\liminf_{n\to\infty}
    \frac{f^{(r)}(n;dk+t,k)}{n^t}
    \ge
    \liminf_{n\to\infty}
    \frac{\xi^{(r)}(n;dk+t,k)}{n^t}
    \ge
    \frac{\gamma(\mathcal{T})}{t!}.$
\end{lemma}

Our probabilistic construction uses the first-moment alteration
method; see~\cite{janson_random_2000}.

\subsection{The lower bound in
\Cref{thm:BES-r-k-t-odd} for odd $k$}\label{subsec:lower-general}

We prove the following strengthened form of the
lower bound in \Cref{thm:BES-r-k-t-odd} for odd $k$.

\begin{theorem}\label{thm:BES-r-k-t-odd-lower}
For all $r>t\ge2$ and odd $k\ge3$, we have 
$\frac{\xi^{(r)}\bigl(n;dk+t,k\bigr)}{n^t}
    \ge
        \frac{2}{t!\bigl(2\binom{r}{t}-1\bigr)}
        -o(1).$
\end{theorem}

The main ingredient is the following proposition, which provides
friendly seeds with packing ratio arbitrarily close to the desired
value.

\begin{prop}\label{prop:exist-friendly}
For all fixed integers $r>t\ge2$ and $\alpha,\beta\ge1$, and every
$\eta>0$, there exists an $(\alpha,\beta)$-friendly $(r,t)$-seed
$\mathcal{T}$ satisfying
$\gamma(\mathcal{T})
    \ge
    \frac{2}{2\binom{r}{t}-1}-\eta$.
\end{prop}

\Cref{thm:BES-r-k-t-odd-lower} follows immediately from \Cref{prop:exist-friendly} (applied with $(\alpha,\beta)=(k,k)$) and \Cref{lemma:friendly-seed-packing}. It remains to prove Proposition~\ref{prop:exist-friendly}.
The proposition follows from the following key lemma.

\begin{lemma}\label{lemma:seed-increment}
Let $r>t\ge2$ and $\alpha,\beta\ge1$ be fixed integers such that
$\beta$ is odd and $\alpha\ge\beta+2$. For every $\eta>0$ and every
$(\alpha,\beta)$-friendly $(r,t)$-seed
$\mathcal{T}=(V,E,J)$, there exists an $(\alpha,\beta+2)$-friendly $(r,t)$-seed
$\mathcal{T}^{+}=(V^{+},E^{+},J^{+})$ such that
$\gamma(\mathcal{T}^{+})
    \ge
    \gamma(\mathcal{T})-\eta.$
\end{lemma}

\begin{proof}[Proof of Proposition~\ref{prop:exist-friendly} assuming \cref{lemma:seed-increment}]
Let $\beta^{\star}$ be the smallest odd integer satisfying
$\beta^{\star}\ge\beta$, and put
$\alpha^{\star}
    =
    \max\{\alpha,\beta^{\star}\}.$
Observe that every
$(\alpha^{\star},\beta^{\star})$-friendly seed is also
$(\alpha,\beta)$-friendly.

Let $E_1$ consist of a single diamond, put $V_1=V(E_1)$, and
let $J_1=\partial_tE_1$. Then
$\mathcal T_1=(V_1,E_1,J_1)$ is $(\alpha^\star,1)$-friendly and
has packing ratio $2/(2\binom rt-1)$.

If $\beta^{\star}=1$, there is nothing to prove.
Otherwise, apply Lemma~\ref{lemma:seed-increment} successively along the chain    $1\longrightarrow3\longrightarrow5
    \longrightarrow\cdots\longrightarrow\beta^{\star},$
always with the fixed parameter $\alpha^{\star}$.
Choose the losses in the finitely many applications so that their sum is
less than $\eta$.
This produces an
$(\alpha^{\star},\beta^{\star})$-friendly seed $\mathcal{T}$ satisfying
$\gamma(\mathcal{T})
    \ge
    \frac{2}{2\binom rt-1}-\eta.$
Since $\alpha^{\star}\ge\alpha$ and
$\beta^{\star}\ge\beta$, the seed $\mathcal{T}$ is also
$(\alpha,\beta)$-friendly.
\end{proof}
We now prove \Cref{lemma:seed-increment} using a probabilistic argument.
We select correlated random copies in a balanced blow-up of the
old seed, so that many copies use the same dangerous $t$-sets.
We then delete copies involved in four types of unwanted configurations,
called conflicts. Sharing the dangerous sets keeps the additional
support small; the alteration preserves the required sparsity.

Fix $\eta>0$.
Throughout the argument, let
$\mathcal T=(V,E,J)$ be the $(\alpha,\beta)$-friendly $(r,t)$-seed
in the statement of \Cref{lemma:seed-increment}.
Recall that $\beta$ is odd and $\alpha\ge\beta+2$.
Put $d=r-t$, $s=|V|$, $e=|E|$, and $j=|J|$, and let
$R=\binom{V}{t}\setminus J$.
Define the family of \emph{dangerous $t$-sets} by
$B=\{p\in R:\mu_p(\mathcal{T})\le\beta+2\}$ and write
$b=|B|$.
If $b=0$, then $\mu_p(\mathcal{T})>\beta+2$ for every
$p\in\binom{V}{t}\setminus J$.
Hence $\mathcal{T}$ itself is $(\alpha,\beta+2)$-friendly, and the
conclusion of Lemma~\ref{lemma:seed-increment} follows by taking
$\mathcal{T}^{+}=\mathcal{T}$.
We may therefore assume that $b>0$.

\begin{fact}\label{fact:dangerous-covered-number}
For every $p\in B$, we have $\mu_p(\mathcal{T})=\beta+1$.
\end{fact}

\begin{proof}
The definitions give $\beta<\mu_p(\mathcal T)\le\beta+2$.
The value $\beta+2$ is impossible: it would give a tight subgraph
with an odd number $\beta+2\le\alpha$ of edges, contrary to
\Cref{def:friendly}(2).
\end{proof}

Choose a constant $\Lambda>0$ sufficiently large such that
$\frac{e}{j+4b/\Lambda}\ge\frac{e}{j}-\eta.$
We will next choose a sufficiently small constant $\delta>0$, put
$c=\Lambda\delta$, and finally choose a sufficiently large integer $N$.
The precise choices of $\delta$ and $N$ will be made later.

For every $x\in V$, let $U_x$ be a set of $N$ vertices, where the sets
$\{U_x\}_{x\in V}$ are pairwise disjoint, and put
$\widetilde{V}=\bigcup_{x\in V}U_x$.
A \emph{transversal map} is a map
$\phi:V\to\widetilde{V}$ satisfying $\phi(x)\in U_x$ for every $x\in V$.
Since the vertex classes are pairwise disjoint, every transversal map is
injective.
For every set $A\subseteq V$, we write
$\phi(A)=\{\phi(x):x\in A\}$.
For every dangerous $t$-set
$p=\{x_1,\ldots,x_t\}\in B$, independently choose an auxiliary random
$t$-partite $t$-uniform hypergraph $G_p$ with vertex classes
$U_{x_1},\ldots,U_{x_t}$ by including each transversal $t$-set
independently with probability $\delta$.
Independently of the hypergraphs $G_p$, choose mutually independent
Bernoulli random variables $Z_\phi$, one for each transversal map,
with parameter $\rho=c\delta^{-b}N^{t-s}$.
Since $s\ge r>t$, we have $\rho\le1$ for all sufficiently large $N$.

We call a transversal map $\phi$ \emph{selected} if $Z_\phi=1$ and
$\phi(p)\in E(G_p)$ for every $p\in B$.
Let $\mathcal{A}$ be the family of selected transversal maps, and put
$X=|\mathcal{A}|$.
For every fixed transversal map $\phi$, the events
$\phi(p)\in E(G_p)$, $p\in B$, are mutually independent and are also
independent of $Z_\phi$.
Consequently,
\begin{align}\label{eq:probability_of_selected}
    \mathbb{P}(\phi\in\mathcal{A})
    =
    \rho\delta^b
    =
    cN^{t-s}.
\end{align}
Since there are exactly $N^s$ transversal maps, it follows from
\eqref{eq:probability_of_selected} and linearity of expectation that
\begin{align}\label{eq:expectation_of_selected}
    \mathbb{E}X
    =
    N^s\cdot cN^{t-s}
    =
    cN^t.
\end{align}

For every transversal map $\phi$, define
$E_\phi=\phi(E)=\{\phi(A):A\in E\}$.
We call $E_\phi$ the copy of $E$ induced by $\phi$.

We now define the configurations that will be removed from
$\mathcal{A}$.
The following multi-hypergraph records coincidences among the images of
dangerous $t$-sets.

\begin{definition}\label{def:junction-hypergraph}
Let $\Phi=(\phi_1,\ldots,\phi_q)$ be an ordered tuple of distinct transversal maps.
For every $p\in B$ and $z\in\binom{\widetilde{V}}{t}$, define
$Q_{\Phi}(p,z)
    =
    \{i\in[q]:\phi_i(p)=z\}$.
Whenever $|Q_{\Phi}(p,z)|\ge2$, we call the labelled set
$Q_{\Phi}(p,z)$ a \emph{junction} of $\Phi$.

The \emph{junction multi-hypergraph} $\mathcal{J}(\Phi)$ is the
multi-hypergraph on vertex set $[q]$ whose hyperedges are all the junctions
of $\Phi$.
Junctions with different labels $(p,z)$ are regarded as distinct
hyperedges, even if they have the same vertex set.
\end{definition}

\begin{definition}[Berge cycle]\label{def:Berge-cycle}
A \emph{Berge cycle} of length $\ell\ge2$ in a hypergraph
$\mathcal H$ is an alternating sequence
$    (v_1,F_1,v_2,F_2,\ldots,v_\ell,F_\ell)$
of distinct vertices $v_1,\ldots,v_\ell$ and distinct hyperedges
$F_1,\ldots,F_\ell$ such that
$\{v_i,v_{i+1}\}\subseteq F_i$ for every $i\in[\ell]$, with indices
taken cyclically. In a multi-hypergraph, different copies of the same
hyperedge are regarded as distinct.

The \emph{incidence graph} of $\mathcal H$ is the bipartite graph
between $V(\mathcal H)$ and $E(\mathcal H)$ defined by containment.
We call $\mathcal H$ \emph{Berge-acyclic} if it contains no
Berge cycle, equivalently, if its incidence graph is a forest.
\end{definition}

\begin{definition}\label{def:bad-conflicts}
A \emph{bad conflict} is one of the following objects.
\begin{enumerate}
    \item
    A \emph{collision conflict} is an ordered pair
    $(\phi_1,\phi_2)$ of distinct selected transversal maps satisfying
        $|\phi_1(V)\cap\phi_2(V)|\ge t+1.$
    An ordered tuple $\Phi=(\phi_1,\ldots,\phi_q)$ of distinct selected
    transversal maps is called \emph{collision-free} if
    $|\phi_i(V)\cap\phi_{i'}(V)|\le t$ for all distinct $i,i'\in[q]$.

    \item
    A \emph{junction-cycle conflict} consists of an ordered tuple
    $\Phi=(\phi_1,\ldots,\phi_q)$ of distinct selected transversal maps,
    where $2\le q\le\alpha$, together with a Berge cycle of length $q$ in
    $\mathcal{J}(\Phi)$.

    \item
    A \emph{supercritical conflict} consists of a collision-free ordered
    tuple $\Phi=(\phi_1,\ldots,\phi_q)$ of distinct selected transversal
    maps, where $q\ge2$, together with nonempty subhypergraphs
    $H_i\subseteq E_{\phi_i}$ for every $i\in[q]$, such that, writing
    $H=\bigcup_{i=1}^q H_i$, we have
    $|H|\le\alpha$ and $v(H)\le d|H|+t-1$.

    \item
    An \emph{exact conflict} consists of a collision-free ordered tuple
    $\Phi=(\phi_1,\ldots,\phi_q)$ of distinct selected transversal maps,
    where $q\ge2$ and $\mathcal{J}(\Phi)$ is Berge-acyclic, together with
    nonempty subhypergraphs $H_i\subseteq E_{\phi_i}$ for every $i\in[q]$,
    such that the following holds. Let $H=\bigcup_{i=1}^q H_i$. We have
    $|H|\le\alpha$, $v(H)=d|H|+t$,
    and at least one of the following conditions holds:
    \begin{enumerate}[label=(\alph*)]
        \item\label{item:small-exact-conflict}
        $|H|\le\beta+2$;

        \item\label{item:odd-exact-conflict}
        $|H|$ is odd.
    \end{enumerate}
\end{enumerate}
\end{definition}

Let $Y$ denote the total number of recorded bad conflicts in
$\mathcal A$, and put $Z=\sum_{p\in B}|E(G_p)|$.
We only need one realization in which the total auxiliary support
is small relative to the number of surviving copies.
The following lemma supplies this; its proof is postponed until
after the proof of \Cref{lemma:seed-increment}.

\begin{lemma}[Random alteration]\label{lemma:good-realization}
After choosing $\delta>0$ sufficiently small and then $N$
sufficiently large, there exists a realization such that
\begin{equation}\label{eq:good-realization}
    X-Y\ge\frac12cN^t,
    \qquad
    Z\le\frac{4b}{\Lambda}(X-Y).
\end{equation}
\end{lemma}

\begin{proof}[Proof of \cref{lemma:seed-increment} assuming \cref{lemma:good-realization}]
The case $b=0$ was settled above, so we may assume that $b>0$.
Fix a realization satisfying the conclusion of
Lemma~\ref{lemma:good-realization}.

For every recorded bad conflict, choose one selected transversal map
participating in that conflict, and delete all maps chosen in this way.
Let $\mathcal{M}$ be the family of surviving transversal maps.
Since at most one map is chosen for each recorded bad conflict, at most
$Y$ distinct maps are deleted. Therefore,
\begin{align}\label{eq:number-surviving-maps}
    |\mathcal{M}|
    \ge X-Y
    \ge \frac12cN^t.
\end{align}
Deleting selected maps cannot create a new bad conflict.
Consequently, $\mathcal{M}$ contains no bad conflict of any of the four
types in Definition~\ref{def:bad-conflicts}.

Define
$E^{+}
    =
    \bigcup_{\phi\in\mathcal{M}}E_{\phi}$,
    $V^{+}
    =
    V(E^{+})$,
and
\[
    J^{+}
    =
    \left(
        \bigcup_{\phi\in\mathcal{M}}\phi(J)
        \,\cup\,
        \bigcup_{\phi\in\mathcal{M}}\phi(B)
    \right)
    \cap\binom{V^{+}}{t}.
\]
Since $\partial_t(E)\subseteq J$, every $t$-subset of an edge in
$E_{\phi}$ belongs to $\phi(J)$. Hence
$\partial_t(E^{+})\subseteq J^{+}$, and therefore
$\mathcal{T}^{+}=(V^{+},E^{+},J^{+})$ is an $(r,t)$-seed.

Since $\mathcal{M}$ contains no collision conflict, any two distinct maps
$\phi,\phi'\in\mathcal{M}$ satisfy
$|\phi(V)\cap\phi'(V)|\le t$.
As $r>t$, the hypergraphs $E_{\phi}$ and $E_{\phi'}$ are edge-disjoint.
It follows that
$|E^{+}|=e|\mathcal{M}|$.

Every selected map sends each $p\in B$ to an edge of $G_p$.
Hence, by \eqref{eq:good-realization} and
$|\mathcal M|\ge X-Y$, we have
$|J^+|\le j|\mathcal M|+Z\le(j+4b/\Lambda)|\mathcal M|$.
Consequently,
\begin{equation}\label{eq:ratio-new-seed}
    \gamma(\mathcal T^+)
    =
    \frac{|E^{+}|}{|J^{+}|}
    \ge\frac{e}{j+4b/\Lambda}
    \ge\gamma(\mathcal T)-\eta.
\end{equation}

It remains to verify that $\mathcal T^+$ is
$(\alpha,\beta+2)$-friendly.
By \cref{def:bad-conflicts} and the absence of bad conflicts
in $\mathcal M$, every ordered tuple
$\Phi=(\phi_1,\ldots,\phi_q)$ of distinct maps in $\mathcal M$
with $q\le\alpha$ is collision-free and has a Berge-acyclic
junction hypergraph $\mathcal J(\Phi)$.
Indeed, a Berge cycle of length $\ell\le q\le\alpha$,
together with the $\ell$ maps corresponding to its vertices,
would form a junction-cycle conflict contained in $\mathcal M$.

Take any nonempty $H\subseteq E^+$ with $|H|\le\alpha$ and
write $H=H_1\cup\cdots\cup H_q$, where
$H_i=H\cap E_{\phi_i}$ is nonempty for each $i$ and
$\phi_1,\ldots,\phi_q\in\mathcal M$ are distinct.
In particular, $q\le|H|\le\alpha$.
If $q=1$, the old seed gives $v(H)\ge d|H|+t$.
If $q\ge2$, the same inequality follows from the absence of
supercritical conflicts. Moreover, if $q\ge2$ and equality holds,
the absence of exact conflicts implies that $|H|$ is even and
$|H|>\beta+2$.
These observations prove the first condition in \Cref{def:friendly}.
They also prove the second: a tight subgraph with an odd number
of at least three edges can neither lie in one copy nor meet
several copies.

Finally, let $z\notin J^+$ be a $t$-set and suppose that a tight
subgraph $H\subseteq E^+$ with $|H|\le\beta+2$ contains $z$ in
its vertex set. The preceding paragraph forces $H$ to lie in a
single copy $E_\phi$. Pulling back through $\phi$ gives
$\mu_{\phi^{-1}(z)}(\mathcal T)\le\beta+2$, so
$\phi^{-1}(z)\in J\cup B$. This implies $z\in J^+$, a contradiction.
Thus the third condition holds as well. Together with
\eqref{eq:ratio-new-seed}, this proves the lemma.
\end{proof}

It remains to prove Lemma~\ref{lemma:good-realization}.
There are two ways a conflict can be rare: it may identify an extra
vertex, saving a factor of $N$, or it may share fewer dangerous sets,
saving a factor of $\delta$. 
The following two parameters record
these savings.

\begin{definition}[Overlap parameters]\label{def:overlap-parameters}
Fix an integer $q\ge2$, and let
$\Phi=(\phi_1,\ldots,\phi_q)$
be an ordered tuple of distinct transversal maps. Recall that
$s=|V|$. Define the total number of vertices used by $\Phi$ by
$v(\Phi)
    :=
    \left|
        \bigcup_{i=1}^q \phi_i(V)
    \right|$,
and define its \emph{vertex overlap} by
$h(\Phi):=qs-v(\Phi)$.

Define the \emph{dangerous-set overlap} of $\Phi$ by
$D(\Phi)
    :=
    \sum_{Q\in E(\mathcal J(\Phi))}
        (|Q|-1)$,
where junctions with different labels are counted separately.
Equivalently,
    $D(\Phi)
    =
    \sum_{p\in B}
    \left(
        q-
        \left|
            \{\phi_i(p):i\in[q]\}
        \right|
    \right).$
\end{definition}

\begin{fact}[Tuple estimate]\label{fact:tuple-estimate}
Fix integers $q\ge2$ and $u,w\ge0$, and suppose that
$0<\delta\le1$. The expected number of selected ordered
$q$-tuples $\Phi$ satisfying
$h(\Phi)\ge u$ and 
$D(\Phi)\le w$
is at most
$C_{q,s}\,c^q\delta^{-w}N^{tq-u}$,
where $C_{q,s}$ depends only on $q$ and $s$.
\end{fact}

\begin{proof}
There are only $O_{q,s}(1)$ ways to specify which of the $qs$
coordinate occurrences among $\phi_1,\ldots,\phi_q$ are equal.
If $h(\Phi)\ge u$, at most $qs-u$ distinct vertices occur,
so there are at most $C_{q,s}N^{qs-u}$ such tuples.

For a fixed tuple $\Phi=(\phi_1,\ldots,\phi_q)$ of distinct maps,
the events $Z_{\phi_i}=1$, $i\in[q]$, contribute a factor $\rho^q$.
For each $p\in B$, we also need every set in
$\{\phi_i(p):i\in[q]\}$ to be an edge of $G_p$.
By the definition of $D(\Phi)$, the total number of these edges,
summed over $p\in B$, is
$\sum_{p\in B}|\{\phi_i(p):i\in[q]\}|=bq-D(\Phi)$.
Each of these edges is included independently with probability
$\delta$, independently also of the variables $Z_{\phi_i}$.
Therefore,
$\mathbb P(\phi_1,\ldots,\phi_q\in\mathcal A)
    =\rho^q\delta^{bq-D(\Phi)}
    =c^q\delta^{-D(\Phi)}N^{q(t-s)}$.
Since $D(\Phi)\le w$ and $\delta\le1$, this is at most
$c^q\delta^{-w}N^{q(t-s)}$.
Multiplying by the tuple count proves the bound.
\end{proof}

\begin{fact}\label{fact:easy-conflict-overlap}
Suppose that a collision conflict, a junction-cycle conflict, or a
supercritical conflict involves the ordered tuple
$\Phi=(\phi_1,\ldots,\phi_q)$
of selected transversal maps. Then
$h(\Phi)\ge tq-t+1$.
\end{fact}

\begin{proof}
For a collision conflict, $q=2$ and
$h(\Phi)=|\phi_1(V)\cap\phi_2(V)|\ge t+1$.

For a junction-cycle conflict, relabel the maps along its Berge
cycle as $1,\ldots,q$. Consecutive maps share a dangerous $t$-set.
The two junctions incident with map $1$ are distinct, so their
shared $t$-sets are distinct.
Indeed, if both shared sets were equal to $z$, then both junctions
would have the label $(\phi_1^{-1}(z),z)$, contradicting their
distinctness.
Add the vertex sets of the maps in
the order $2,3,\ldots,q,1$. Each intermediate set overlaps the
preceding union in at least $t$ vertices, and the last overlaps
it in at least $t+1$. Hence
$h(\Phi)\ge(q-2)t+(t+1)=tq-t+1$, also when $q=2$.

For a supercritical conflict, the $H_i$ are edge-disjoint and the
old seed gives $v(H_i)\ge d|H_i|+t$ for every $i$, including
single edges. Therefore
\[
    h(\Phi)\ge\sum_i v(H_i)-v(H)
    \ge d\sum_i|H_i|+qt-(d|H|+t-1)=tq-t+1.\qedhere
\]
\end{proof}

\begin{fact}\label{fact:acyclic-junction-D}
Let $\Phi=(\phi_1,\ldots,\phi_q)$ be an ordered tuple of distinct
transversal maps. If $\mathcal J(\Phi)$ is Berge-acyclic, then
$D(\Phi)\le q-1$.
Moreover, if $D(\Phi)=q-1$, then the incidence graph of
$\mathcal J(\Phi)$ is connected and hence is a tree. In particular,
every $i\in[q]$ belongs to some junction of $\mathcal J(\Phi)$.
\end{fact}

\begin{proof}
Since $\mathcal J(\Phi)$ is Berge-acyclic, its incidence graph is a
forest. Let $m$ be the number of junctions, counted with their labels,
and let $c_0$ be the number of connected components of the incidence
graph, including the isolated vertices in $[q]$. The incidence graph has
$q+m$ vertices and hence $q+m-c_0$ edges. Therefore,
    $D(\Phi)
    =
    \sum_{Q\in E(\mathcal J(\Phi))}(|Q|-1)
    =
    (q+m-c_0)-m
    =
    q-c_0
    \le q-1.$
If $D(\Phi)=q-1$, then $c_0=1$. Thus the incidence graph is connected
and hence is a tree. In particular, no vertex in $[q]$ is isolated, so
every $i\in[q]$ belongs to some junction of $\mathcal J(\Phi)$.
\end{proof}

The only potentially costly case has the minimum vertex overlap
$h(\Phi)=tq-t$ and a tree of dangerous-set overlaps, so that
$D(\Phi)=q-1$. Equality forces every participating subgraph to
cover a dangerous set. 
Each must then have an even number of edges
and at least $\beta+1$ edges. This is precisely why such a union
cannot be one of the exact conflicts we need to remove.

\begin{fact}\label{fact:critical-exact-overlap}
Suppose that an exact conflict consists of an ordered tuple
$\Phi=(\phi_1,\ldots,\phi_q)$
and nonempty subhypergraphs
$H_i\subseteq E_{\phi_i}$ for every $i\in[q]$. 
If $h(\Phi)=tq-t$ and $D(\Phi)=q-1$,
then $|H|$ is even and $|H|>\beta+2$. Consequently, no exact conflict
can satisfy both equalities.
\end{fact}

\begin{proof}
Put $e_i=|H_i|$ and $v_i=v(H_i)$. The copies are edge-disjoint,
and the old seed gives $v_i\ge de_i+t$, including when $e_i=1$.
Since $v(H)=d\sum_i e_i+t$, we have    $h(\Phi)\ge\sum_i v_i-v(H)\ge t(q-1)=h(\Phi).$
Thus equality holds throughout. In particular,
\begin{equation*}\label{eq:critical-overlap-equalities}
    v_i=de_i+t\quad\text{for every }i,
    \qquad \sum_i v_i-v(H)=h(\Phi).
\end{equation*}
The second equality implies that every shared vertex is retained
in all the corresponding sets $V(H_i)$.
Indeed, a vertex belonging to exactly $a\ge2$ of the sets
$\phi_i(V)$ contributes $a-1$ to $h(\Phi)$.
If it belonged to only $a'<a$ of the corresponding sets $V(H_i)$,
its contribution to $\sum_i v_i-v(H)$ would be
$\max\{a'-1,0\}<a-1$.
Since $V(H_i)\subseteq\phi_i(V)$ for every $i$, the contribution
of any other vertex cannot increase, contradicting the second equality.

Since $D(\Phi)=q-1$, \Cref{fact:acyclic-junction-D} shows that
every map belongs to a junction $Q_\Phi(p_i,z_i)$.
All vertices of $z_i=\phi_i(p_i)$ are shared, so $z_i\subseteq V(H_i)$.
Consequently,
$e_i\ge\mu_{p_i}(\mathcal T)=\beta+1$ by
\Cref{fact:dangerous-covered-number}.
As $e_i\ge2$, tightness and \Cref{def:friendly}(2) force $e_i$
to be even. Hence $|H|=\sum_i e_i$ is even and
$|H|\ge2(\beta+1)>\beta+2$, contradicting both alternatives
in the definition of an exact conflict.
\end{proof}

The same vertex-counting argument used above always gives
$h(\Phi)\ge tq-t$. Since $h(\Phi)$ is an integer, either
$h(\Phi)\ge tq-t+1$ or $h(\Phi)=tq-t$. In the latter case,
Fact~\ref{fact:acyclic-junction-D} gives $D(\Phi)\le q-1$, while
Fact~\ref{fact:critical-exact-overlap} excludes $D(\Phi)=q-1$.
Thus, we have now obtained the following dichotomy. If an exact conflict
involves an ordered $q$-tuple $\Phi$ of selected transversal maps, then
\[ \text{either }
h(\Phi)\ge tq-t+1,
\qquad \text{or} \qquad 
    h(\Phi)=tq-t
    \quad\text{and}\quad
    D(\Phi)\le q-2.
\]

We are now ready to prove Lemma~\ref{lemma:good-realization}.

\begin{proof}[Proof of Lemma~\ref{lemma:good-realization}]
Let $Y_{\mathrm{easy}}$ count the collision, junction-cycle, and
supercritical conflicts, together with the exact conflicts whose
$q$-tuple satisfies $h(\Phi)\ge tq-t+1$.
Every such tuple satisfies $2\le q\le\alpha$ and
$D(\Phi)\le b(q-1)$.
Moreover, each tuple gives rise to only a bounded number of
conflicts, with the bound depending only on $\mathcal T$ and $\alpha$.
By \Cref{fact:easy-conflict-overlap,fact:tuple-estimate},
\begin{equation}\label{eq:expected-easy-conflicts}
    \mathbb E Y_{\mathrm{easy}}
    =O_{\mathcal T,\alpha,c,\delta}(N^{t-1}).
\end{equation}

For each remaining conflict, the preceding dichotomy gives
$h(\Phi)=tq-t$ and $D(\Phi)\le q-2$.
For each fixed $q$, \Cref{fact:tuple-estimate} bounds the expected
number of such conflicts by a constant times
$c^q\delta^{-(q-2)}N^t=\Lambda^q\delta^2N^t$.
Summing over $2\le q\le\alpha$ yields
\begin{equation}\label{eq:expected-critical-exact-conflicts}
    \mathbb E(Y-Y_{\mathrm{easy}})\le K\delta^2N^t,
\end{equation}
where $K=K(\mathcal T,\alpha,\Lambda)$ is independent of $\delta,N$.
Choose $0<\delta\le\min\{1,\Lambda^{-1}\}$ so that
$K\delta^2\le c/8$, and then choose $N$ large enough that
$\mathbb E Y_{\mathrm{easy}}\le cN^t/8$.
Thus $\mathbb E Y\le cN^t/4$.

We already have $\mathbb E X=cN^t$, while linearity of expectation
gives $\mathbb E Z=b\delta N^t$.
Since $c=\Lambda\delta$,
\[
    \mathbb E\left[X-Y-\frac{\Lambda}{4b}Z\right]
    \ge cN^t-\frac14cN^t-\frac14cN^t
    =\frac12cN^t.
\]
Choose a realization in which the expression in brackets is at
least its expectation. Since $Z\ge0$, this realization satisfies
both inequalities in \eqref{eq:good-realization}.
\end{proof}

\subsection{Proof of \Cref{thm:BES-3-2-even}}
In this subsection, we prove \cref{thm:BES-3-2-even}. We construct a seed for each even $k$ and apply
\Cref{lemma:glock-packing}. 
The construction keeps a positive excess
in $3|E|-|J|$ while increasing the non-edge girth.
\begin{definition}\label{def:seed-3-2-even+}
For every integer $\ell\ge1$, we recursively define a
\((3,2)\)-seed
$\mathcal S_\ell=(V_\ell,E_\ell,J_\ell)$
as follows.
The initial seed $\mathcal S_1$ is the seven-edge configuration
used in~\cite{glock_64_2024}
; see \Cref{fig:initial-seed-3-2-even}.
It is defined by
$V_1=\{a,b,c,x_1,y_1,x_2,y_2,x_3,y_3\}$,
$
E_1=
\{
abc,\,
ax_1y_1,\,
bx_1y_1,\,
bx_2y_2,\,
cx_2y_2,\,
cx_3y_3,\,
ax_3y_3
\}$ and 
$J_1 = \partial_2(E_1).$ Now suppose that \(\mathcal S_\ell=(V_\ell,E_\ell,J_\ell)\)
has been defined for some \(\ell\ge1\), and put
$R_\ell=\binom{V_\ell}{2}\setminus J_\ell$.
For every \(p=\{u,v\}\in R_\ell\), choose a new two-element set
$X_p=\{x_p,y_p\}$,
such that these sets are pairwise disjoint and disjoint from
\(V_\ell\).
Define
$\mathcal D_p
=
\bigl\{
\{u\}\cup X_p,\,
\{v\}\cup X_p
\bigr\}$.
We then set
\begin{align*}
    V_{\ell+1}
=
V_\ell\cup\bigcup_{p\in R_\ell}X_p, \quad
E_{\ell+1}
=
E_\ell\cup\bigcup_{p\in R_\ell}\mathcal D_p,
\quad
J_{\ell+1}
=
J_\ell\cup
\bigcup_{p\in R_\ell}
\binom{p\cup X_p}{2}.
\end{align*}

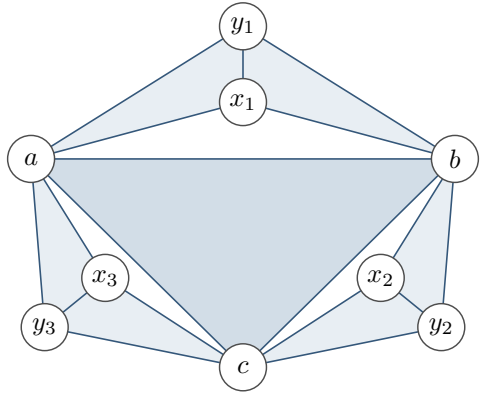
\begin{figure}[t]
\centering
\begin{tikzpicture}[
  line cap=round,
  line join=round,
  vertex/.style={
    draw=black!70,
    circle,
    line width=0.55pt,
    fill=white,
    minimum size=6.2mm,
    inner sep=0pt,
    font=\small
  },
  centraledge/.style={
    fill=seedblue!24,
    draw=none
  },
  outeredge/.style={
    fill=seedblue!12,
    draw=none
  },
  Jpair/.style={
    draw=seedblue!78!black,
    line width=0.60pt
  }
]

\coordinate (a) at (-2.80, 0.70);
\coordinate (b) at ( 2.80, 0.70);
\coordinate (c) at ( 0.00,-2.05);

\coordinate (x1) at ( 0.00, 1.45);
\coordinate (y1) at ( 0.00, 2.45);

\coordinate (x2) at ( 1.82,-0.87);
\coordinate (y2) at ( 2.62,-1.52);

\coordinate (x3) at (-1.82,-0.87);
\coordinate (y3) at (-2.62,-1.52);

\path[centraledge] (a)--(b)--(c)--cycle;

\path[outeredge] (a)--(x1)--(y1)--cycle;
\path[outeredge] (b)--(x1)--(y1)--cycle;

\path[outeredge] (b)--(x2)--(y2)--cycle;
\path[outeredge] (c)--(x2)--(y2)--cycle;

\path[outeredge] (c)--(x3)--(y3)--cycle;
\path[outeredge] (a)--(x3)--(y3)--cycle;

\draw[Jpair] (a)--(b);
\draw[Jpair] (a)--(x1);
\draw[Jpair] (a)--(y1);
\draw[Jpair] (b)--(x1);
\draw[Jpair] (b)--(y1);
\draw[Jpair] (x1)--(y1);

\draw[Jpair] (b)--(c);
\draw[Jpair] (b)--(x2);
\draw[Jpair] (b)--(y2);
\draw[Jpair] (c)--(x2);
\draw[Jpair] (c)--(y2);
\draw[Jpair] (x2)--(y2);

\draw[Jpair] (c)--(a);
\draw[Jpair] (c)--(x3);
\draw[Jpair] (c)--(y3);
\draw[Jpair] (a)--(x3);
\draw[Jpair] (a)--(y3);
\draw[Jpair] (x3)--(y3);

\node[vertex] at (a)  {\(a\)};
\node[vertex] at (b)  {\(b\)};
\node[vertex] at (c)  {\(c\)};

\node[vertex] at (x1) {\(x_1\)};
\node[vertex] at (y1) {\(y_1\)};
\node[vertex] at (x2) {\(x_2\)};
\node[vertex] at (y2) {\(y_2\)};
\node[vertex] at (x3) {\(x_3\)};
\node[vertex] at (y3) {\(y_3\)};

\end{tikzpicture}
\caption{The initial seed \(\mathcal S_1\).}
\label{fig:initial-seed-3-2-even}
\end{figure}
\end{definition}

We begin with the following properties for $\mathcal{S}_{\ell}$.

\begin{prop}\label{prop:properties-seed-ell}
For any integer \(\ell\ge1\), we have the following.
\begin{enumerate}
    \item We have
$        \frac{|E_\ell|}{|J_\ell|}
        =
        \frac{1}{3}+\frac{1}{|J_\ell|}
        >
        \frac{1}{3}.$

    \item For any integer \(h\ge1\), the \(3\)-graph
    \((V_\ell,E_\ell)\) is \((h+1,h)\)-free.

    \item If \(H\subseteq E_\ell\) satisfies
    $v(H)
        \le e(H)+2$,
    then either \(e(H)\leq 2\), or
    \(abc\in H\) and \(e(H)\) is odd.

    \item For any \(p\in R_\ell\), we have
    $\mu_p(\mathcal S_\ell)\ge 2\ell+1$.
\end{enumerate}
\end{prop}

\begin{proof}[Proof of \Cref{prop:properties-seed-ell}]
Initially $|E_1|=7$ and $|J_1|=18$.
Each added diamond contributes two edges and six new reserved
pairs. Hence $3|E_\ell|-|J_\ell|=3$ for every $\ell$, proving
$|E_\ell|/|J_\ell|=1/3+1/|J_\ell|$.

We prove the remaining assertions together by induction on $\ell$.
The base case follows by inspecting the seven edges of $E_1$;
see \Cref{fig:initial-seed-3-2-even}.
For a nonempty $H\subseteq E_{\ell+1}$, put $H_0=H\cap E_\ell$.
Let $\mathcal P$ be the set of pairs $p\in R_\ell$ for which
$H\cap\mathcal D_p\ne\varnothing$, and let $\mathcal P_1$ contain
those contributing just one edge. Let $W$ be the set of old
vertices used by these new edges.
Since the new vertex pairs $X_p$ are pairwise disjoint and
disjoint from $V_\ell$, each $p\in\mathcal P$ contributes exactly
two new vertices to $H$.
The number of new edges in $H$ is
$2|\mathcal P|-|\mathcal P_1|$.
Therefore,
\begin{equation}\label{eq:H_0-and-H}
    \begin{aligned}
    v(H)-e(H)
    &= v(H_0)+|W\setminus V(H_0)|+2|\mathcal P|
       -\bigl(e(H_0)+2|\mathcal P|-|\mathcal P_1|\bigr)\\
    &= v(H_0)-e(H_0)
       +|W\setminus V(H_0)|+|\mathcal P_1|.
\end{aligned}
\end{equation}

If $H_0=\varnothing$, the right-hand side is
$|W|+|\mathcal P_1|\ge2$. 
Equality forces either $|W|=|\mathcal P_1|=1$, in which case
$H$ consists of a single edge, or $|W|=2$ and
$\mathcal P_1=\varnothing$, in which case $H$ consists of a
single whole diamond.
In either case every pair in $V(H)$ belongs to $J_{\ell+1}$.
If $H_0\ne\varnothing$, induction and \eqref{eq:H_0-and-H} give
$v(H)-e(H)\ge2$ again. This proves the second assertion.
Moreover, equality in this case holds precisely when $H_0$ is
tight and all new edges come in whole diamonds with their
endpoint pairs contained in $V(H_0)$.

Suppose that $H$ is tight and has at least three edges.
Then $H_0\ne\varnothing$. If no new diamond is used, apply induction
to $H_0=H$. Otherwise, for each used pair $p\in\mathcal P$, the
fourth induction hypothesis gives
$e(H_0)\ge\mu_p(\mathcal S_\ell)\ge2\ell+1\ge3$.
Thus $H_0$ contains $abc$ and has odd size by the third induction
hypothesis. Adding whole diamonds preserves both properties,
proving the third assertion.

Finally, let $z\in R_{\ell+1}$ lie in the vertex set of a tight
$H\subseteq E_{\ell+1}$. All pairs within $V_\ell$ belong to
$J_{\ell+1}$, so $z$ uses a new vertex in some $X_p$.
If $H_0=\varnothing$, then $H$ is a single edge or a single
whole diamond, so every pair in $V(H)$ belongs to $J_{\ell+1}$,
contradicting $z\in R_{\ell+1}$.
Thus $H_0\ne\varnothing$, and the equality characterization above
shows that $\mathcal D_p\subseteq H$ and that the tight subgraph
$H_0$ contains $p$ in its vertex set. Therefore
$e(H)\ge e(H_0)+2\ge\mu_p(\mathcal S_\ell)+2\ge2\ell+3$.
This proves the fourth assertion and completes the induction.
\end{proof}

It remains to prove \Cref{thm:BES-3-2-even}.

\begin{proof}[Proof of \Cref{thm:BES-3-2-even}]
Let $k\ge4$ be even and put $\ell=\lceil k/4\rceil$.
By the third assertion in
\Cref{prop:properties-seed-ell}, the seed $\mathcal S_\ell$ is
$(k+2,k)$-free: otherwise, a tight $k$-edge subhypergraph would
force $k$ to be odd.
By the second assertion, it is also $(k_1+1,k_1)$-free for every
integer $2\le k_1\le k-1$.
Moreover, the fourth assertion gives that the girth of
$\mathcal S_\ell$ is at least $2\ell+1>k/2$. Therefore, by \Cref{lemma:glock-packing} and the first assertion of \Cref{prop:properties-seed-ell},
$    \liminf_{n\to\infty}
    \frac{f^{(3)}(n;k+2,k)}{n^2}
    \ge        \frac{\gamma(\mathcal S_\ell)}{2}
    =
    \frac{1}{6}
    +
    \frac{1}{2|J_\ell|}
    >
    \frac{1}{6}.$
\end{proof}

\section{Proof of \Cref{thm:generalized-ramsey}}\label{sec:ramsey}
In this section, we prove \Cref{thm:generalized-ramsey} using
the following two comparison lemmas. These lemmas are direct generalizations of the lemmas used in~\cite{bennett_cushman_dudek_2025}.
Throughout this section, we let $p=r(k+1)$ and
$q=\binom{p}{r}-k+1$,
where $r,k\ge2$. 

\begin{lemma}\label{lemma:generalized-ramsey-upperbound}
For all $r,k\ge2$ and $n\ge p$, we have
$    \operatorname{ES}^{(r)}
    \bigl(n;p,q\bigr)
    \le
    \binom{n}{r}
    -
    \xi^{(2r)}
    \bigl(n;p,k\bigr).$
Consequently,
$    \limsup_{n\to\infty}
    \frac{
        \operatorname{ES}^{(r)}
        \bigl(n;p,q\bigr)
    }{n^r}
    \le
    \frac{1}{r!}
    -
    \liminf_{n\to\infty}
    \frac{
        \xi^{(2r)}
        \bigl(n;p,k\bigr)
    }{n^r}.$
\end{lemma}

\begin{lemma}\label{lemma:generalized-ramsey-lowerbound}
For all $r,k\ge2$ and $n\ge p$, we have
$    \operatorname{ES}^{(r)}
    \bigl(n;p,q\bigr)
    \ge
    \binom{n}{r}
    -
    f^{(2r)}
    \bigl(n;p,k\bigr)
    -
    (k-2).$
Consequently,
$    \liminf_{n\to\infty}
    \frac{
        \operatorname{ES}^{(r)}
        \bigl(n;p,q\bigr)
    }{n^r}
    \ge
    \frac{1}{r!}
    -
    \limsup_{n\to\infty}
    \frac{
        f^{(2r)}
        \bigl(n;p,k\bigr)
    }{n^r}.$
\end{lemma}

\begin{proof}[Proof of \Cref{thm:generalized-ramsey} assuming \cref{lemma:generalized-ramsey-lowerbound,lemma:generalized-ramsey-upperbound}]
By \Cref{thm:BES-r-k-t-odd,thm:BES-r-k-t-odd-lower} and \eqref{eq:xi-packing-lower}, we have
\begin{align*}
    \lim_{n\to\infty}
    \frac{f^{(2r)}(n;p,k)}{n^r}
    =
    \lim_{n\to\infty}
    \frac{\xi^{(2r)}(n;p,k)}{n^r}
    =
    \begin{cases}
        \displaystyle
        \frac{2}{r!\left(2\binom{2r}{r}-1\right)},
        & \text{if $k$ is odd},\\[1.2em]
        \displaystyle
        \frac{1}{r!\binom{2r}{r}},
        & \text{if $k$ is even}.
    \end{cases}
\end{align*}
Combining
\Cref{lemma:generalized-ramsey-upperbound,lemma:generalized-ramsey-lowerbound}, we complete the proof.
\end{proof}
It remains to prove the two lemmas.
\begin{proof}[Proof of \Cref{lemma:generalized-ramsey-upperbound}]
Let $H$ attain $\xi^{(2r)}(n;p,k)$.
It is $(p,k)$-free and $(r(i+1)-1,i)$-free for $2\le i<k$,
so in particular it contains no $(k,b)$-block.
The component argument in the proof of
\Cref{thm:BES-r-k-t-odd-upper}, with $(r,t)$ replaced by $(2r,r)$,
shows that each $r$-component has fewer than $k$ edges and admits
an ordering $F_1,\ldots,F_m$ in which each $F_i$, $i\ge2$,
has exactly $r$ vertices outside $F_1\cup\cdots\cup F_{i-1}$.
To see the latter assertion, order the edges so that each $F_i$,
$i\ge2$, is $r$-adjacent to an earlier edge, and put
$U_i=\bigcup_{j=1}^i F_j$.
This gives $|U_i|\le r(i+1)$, while the
$(r(i+1)-1,i)$-freeness gives $|U_i|\ge r(i+1)$
for $2\le i\le m<k$.
Since $|U_1|=2r$, we have $|U_i|=r(i+1)$ for every $i$,
and hence $|F_i\setminus U_{i-1}|=r$ for $i\ge2$.

Choose a partition $F=A_F\mathbin{\dot\cup}B_F$ into two $r$-sets
for every edge, so that all chosen parts are distinct.
For the first edge of each component, choose any partition.
For every subsequent edge, choose both parts to contain a new
vertex; this is possible because there are $r\ge2$ new vertices.
Neither part was chosen earlier. Parts chosen in different
components cannot coincide, since the corresponding edges would
then be $r$-adjacent.

For each $F\in H$, give $A_F,B_F$ a common color used nowhere else,
and give every remaining $r$-set its own new color.
This uses $\binom nr-|H|$ colors.
For any $S\in\binom{[n]}p$, a color is repeated inside $S$ precisely
when its corresponding edge $F=A_F\cup B_F$ lies in $S$.
Since $H$ is $(p,k)$-free, there are at most $k-1$ such repetitions.
Thus $S$ receives at least $\binom pr-(k-1)=q$ colors, proving
$\operatorname{ES}^{(r)}(n;p,q)\le\binom nr-\xi^{(2r)}(n;p,k)$.
\end{proof}

\begin{proof}[Proof of \Cref{lemma:generalized-ramsey-lowerbound}]
Let $\chi$ be a $(p,q)$-coloring of $K_n^{(r)}$ using $m$ colors.
For every color $c$ used by $\chi$, choose an $r$-set $A_c$ of
color $c$. For every other $r$-set $A$ of color $c$, choose a set
$F_{c,A}\in\binom{[n]}{2r}$ such that
$A_c\cup A\subseteq F_{c,A}$.
Such a set exists since $n\ge p=r(k+1)\ge2r$.
Let $\mathcal M$ be the $2r$-uniform multi-hypergraph consisting
of one labelled copy of $F_{c,A}$ for every such pair $(c,A)$.
Different pairs are retained as different copies even if they
produce the same $2r$-set. We then have
$    |\mathcal M|
    =
    \binom{n}{r}-m.$

\begin{claim}\label{claim:ramsey-M-free}
    $\mathcal M$ is $(p,k)$-free, where multiplicities
are counted.
\end{claim}

\begin{poc}
Suppose that $k$ labelled members have union contained in some
$S\in\binom{[n]}p$. If $m_c$ of these members arise from color
$c$, then $S$ contains $A_c$ and $m_c$ distinct additional $r$-sets
of that color. Thus the total number of color repetitions in $S$
is at least $\sum_c m_c=k$, leaving at most
$\binom pr-k=q-1$ colors, a contradiction.
\end{poc}

Let $\mathcal R$ consist of all members of $\mathcal M$ whose
support has multiplicity at least two. We claim that
$|\mathcal R|\le k-1$. Otherwise, take supports one at a time
until their copies include at least $k$ members. Each support
contributes at least two copies, so at most $\lceil k/2\rceil$
supports are needed. Choosing $k$ of these members gives a union
of size at most $2r\lceil k/2\rceil\le r(k+1)=p$, a contradiction.

Let $H$ be the underlying simple $2r$-graph of $\mathcal M$.
If $\mathcal R$ is nonempty, retaining one copy of each repeated
support discards at most $|\mathcal R|-1\le k-2$ members; otherwise
none are discarded. Since $H$ is $(p,k)$-free,
\[
    \binom nr-m=|\mathcal M|
    \le |H|+k-2\le f^{(2r)}(n;p,k)+k-2.
\]
Rearranging proves the lemma.
\end{proof}

\section{Concluding remarks}\label{sec:conclusion}

The integer-exponent Brown--Erd\H{o}s--S\'os problem now has a
uniform answer outside one family: odd $k$ gives the diamond
coefficient, and even $k$ gives the Steiner packing coefficient
whenever $r\ge4$. 
The remaining exact values concern triple systems
with an even number of forbidden edges. Our strict lower bound and the upper bound in
\Cref{thm:BES-r-k-t-odd-upper}, together with
\Cref{thm:f-to-g-cleanup}, give
$\frac16<\pi(3,2,2\ell)\le\frac15$
for every $\ell\ge2.$
Thus the exceptional behavior persists for every even $k$.

\begin{ques}\label{ques:determine-32k}
Determine $\pi(3,2,2\ell)$ for $\ell\ge4$.
\end{ques}

\bibliographystyle{abbrv}
\bibliography{BES_reference}
\end{document}